\documentclass[11pt,
]{amsart}
\usepackage{
amssymb, graphicx
}
\usepackage{mathrsfs} 
\usepackage{amsthm}
\usepackage{bm}
\usepackage{graphicx,color}
\usepackage[dvipsnames]{xcolor}
\usepackage{amsmath}
\allowdisplaybreaks[4]
\newtheorem{theorem}{Theorem}[section]
\newtheorem{lemma}[theorem]{Lemma}
\newtheorem{proposition}[theorem]{Proposition}

\theoremstyle{remark}

\numberwithin{equation}{section}
\date{\today}

\newcommand{\norm}[1]{\lVert #1 \rVert}

\def\R{\mathbb R}

\title[Determination of an anisotropic perturbation]{Determination of an anisotropic perturbation in elastic inverse scattering}

\author[M.~Lassas]{Matti Lassas}
\address{Department of Mathematics and Statistics, University of Helsinki, Helsinki, Finland}
\email{matti.lassas@helsinki.fi}

\author[S.~Ma]{Shiqi Ma}
\address{School of Mathematics, Jilin University, Changchun 130012, China}
\email{mashiqi@jlu.edu.cn}

\author[L.~Oksanen]{Lauri Oksanen}
\address{Department of Mathematics and Statistics, University of Helsinki, Helsinki, Finland}
\email{lauri.oksanen@helsinki.fi}

\author[M.~Salo]{Mikko Salo}
\address{Department of Mathematics and Statistics, University of Jyv\"askyl\"a, Jyv\"askyl\"a, Finland}
\email{mikko.j.salo@jyu.fi}

\author[J. Zhai]{Jian Zhai}
\address{School of Mathematical Sciences,
  Fudan University, Shanghai 200433, China; 
Center for Applied Mathematics,
  Fudan University, Shanghai 200433, China}
 \email{jianzhai@fudan.edu.cn}
\date{}

\begin{document}

\begin{abstract}
We consider a linearized inverse scattering problem for elastic waves. We prove that a fully anisotropic perturbation of the elastic parameters around an isotropic and homogeneous reference can be uniquely determined by (single-)scattered waves. We also give a quantitative stability estimate for an isotropic perturbation, and as a consequence a rigidity result is established.
\end{abstract}


\maketitle

\section{Introduction}

The elastic wave equation describes the propagation of mechanical disturbances in solid media, forming a fundamental model in seismology, geophysics, and material science \cite{MarsdenHughes1994}.
The equation reads
\begin{equation} \label{fullelasticWE}
    \rho \partial_t^2 u_i -\partial_j (c_{ijk\ell} \partial_ku_\ell) = 0 \quad \text{in } \mathbb{R}^3 \times \mathbb{R},
\end{equation}
where $\rho$ is the density and $\mathbf{C}=(c_{ijk\ell}(x))$ is the elastic tensor satisfying the following symmetry properties
\[
c_{ijk\ell}=c_{jik\ell}=c_{k\ell ij}.
\]
Assume that there exists a $c_0>0$ such that $\rho(x)>c_0$ and  
\begin{equation}\label{ellipticitycondition}
\sum_{i,j,k,\ell=1}^3c_{ijk\ell}(x)\varepsilon_{ij}\varepsilon_{k\ell}\geq c_0\sum_{i,j=1}^3\varepsilon^2_{ij}
\end{equation}
for any nonzero $3\times 3$ real-valued symmetric matrix $(\varepsilon_{ij})$.

Further we assume that $\mathbf{C},\rho\in C^\infty(\mathbb{R}^3)$ and there are three constants $\lambda_0,\mu_0,\rho_0$ such that quantities $C_{ijk\ell}-\lambda_0 \delta_{ij} \delta_{k\ell}-\mu_0(\delta_{ik} \delta_{j\ell} + \delta_{i\ell} \delta_{jk})$ and $\rho-\rho_0$ are compactly supported in $\Omega:=\{x\in\mathbb{R}^3:|x|<1\}$. Throughout the paper we assume $\rho_0=1$ without loss of any generality.

Let $u_0$ be a solution to \eqref{fullelasticWE} with $\mathbf{C} = \mathbf{C}_0 :=(\lambda_0\delta_{ij}\delta_{k\ell}+\mu_0(\delta_{ik} \delta_{j\ell}+\delta_{i\ell} \delta_{jk}))$ and $\rho=1$ (the medium is isotropic and homogeneous), that is,
\begin{equation} \label{isotropicEWE}
\partial^2_tu_0-(\lambda_0+\mu_0)\nabla\nabla\cdot u_0-\mu_0\Delta u_0=0.
\end{equation}
The ellipticity condition \eqref{ellipticitycondition} reduces to
\[
\mu_0>0,\quad 3\lambda_0 + 2\mu_0>0.
\]
In this study we only use special solutions $u_0$ to \eqref{isotropicEWE}. Denote
\[
c_p=\sqrt{\lambda_0+2\mu_0},\quad c_s=\sqrt{\mu_0},
\]
to be the $P$-wave and $S$-wave speeds respectively. Note that $c_p>c_s$. Notice that we can take a plane $P$-wave solution
\[
u_0=\alpha H_0(c_pt-\theta\cdot x),
\]
where $\alpha=\theta$,
or a plane $S$-wave solution
\[
u_0=\alpha H_0(c_st-\theta\cdot x),
\]
where $\theta\in\mathbb{S}^2$ and $\alpha\perp\theta$. Here $H_0$ is the Heaviside function. It is easy to verify that they both satisfy equation \eqref{isotropicEWE}.
By the standard hyperbolic theory (cf., for example, \cite[Chapter 23]{hormander2007analysis}), there exists a unique distributional solution $U(x,t;\theta,\alpha)$, where $\alpha=\theta$  or $\alpha\perp\theta$, to the equation
\begin{equation} \label{scatteredUeq}
\begin{split}
\rho\partial_t^2U_i -\partial_j (c_{ijk\ell} \partial_kU_\ell) & = 0, \quad \,\,\, \text{in } \mathbb{R}^3 \times \mathbb{R}, \\
U & = u_0, \quad \text{for }t\ll 0.
\end{split}
\end{equation}
This model describes the wave scattering caused by the inhomogeneity encoded in $(\mathbf{C},\rho)$ with incident plane wave $u_0$ in the isotropic and homogeneous surrounding background.

We consider the inverse scattering problem associated with the above equation. In this setting, $u_0$ represents the incident wave, and from the total wave field $U$ measured outside $\Omega$ we aim to recover the elastic tensor $\mathbf{C}$ and the density $\rho$.
In general, we are interested in the invertibility of the forward map
\[
\mathcal{A} \colon (\mathbf{C},\rho) \mapsto \big\{ U(\cdot,\cdot;\theta,\alpha) \vert_{\Omega^c \times (-\infty,T]} \big\vert \theta \in \mathbb{S}^2,\alpha=\theta \text{ or } \alpha \perp \theta \big\}
\]
with $T>0$ sufficiently large, where $U$ satisfies the equation \eqref{scatteredUeq}. When $\mathbf{C}$ is isotropic, inverse boundary value problems have been studied  \cite{romanov2002inverse,rachele2000inverse,rachele2003uniqueness,stefanov2017local,bhattacharyya2018local,zhai2025determination}, while the fully anisotropic problem is far from reach. When the elastic tensor is piecewise analytic, certain types of anisotropy can be dealt with \cite{de2019unique}. While the Boundary Control method, invented by Belishev \cite{belishev1987approach}, can be used to study the anisotropic acoustic wave equation \cite{belishev1992reconstruction}, it does not work well for the elastic wave equation even for the isotropic case \cite{belishev2007recent}. We also refer to \cite{katchalov1998multidimensional,kachalov2001inverse,kurylev2006maxwell,lassas2014inverse} for further development of the Boundary Control method.

While the fully nonlinear problem, that is, the injectivity of $\mathcal{A}$, is widely open, in this article we consider the invertibility of the Frech\'et derivative of $\mathcal{A}$ at $(\mathbf{C}_0,1)$, denoted as $\mathcal{A}'[\mathbf{C}_0,1]$. This problem is sometimes called the Born approximation in physical literature. Denote $\mathbf{\dot{C}}=(\dot{c}_{ijk\ell})$ and $u_0=(u_{01},u_{02},u_{03})$. One can check that $\mathcal{A}'[\mathbf{C}_0,1](\mathbf{\dot{C}},\dot\rho)=\dot{u}\vert_{\Omega^c\times(-\infty,T]}$, where $\dot{u}$ is the solution to
\begin{equation*}
    \begin{aligned}
        \partial^2_t\dot{u}_i -(\lambda_0+\mu_0) \nabla_i \nabla \cdot \dot{u}-\mu_0\Delta \dot{u}_i & = \partial_j(\dot c_{ijk\ell}\partial_ku_{0\ell})-\dot\rho\partial^2_tu_{0i}, && \text{in } \mathbb{R}^3 \times \mathbb{R}, \\
        \dot{u} & = 0, && \text{for }t\ll 0.
    \end{aligned}
\end{equation*}
Here by our assumption, $(\mathbf{\dot C},\dot\rho)$ vanishes outside $\Omega$.
The perturbation of the elastic tensor $\dot c_{ijk\ell}$ inherits the symmetry conditions
\[
\dot c_{ijk\ell}=\dot c_{jik\ell}=\dot c_{k\ell ij}.
\]
We study the invertibility of $\mathcal{A}'[\mathbf{C}_0,1]$ allowing the perturbation $\mathbf{\dot C}$ to be fully anisotropic. The main result of this paper is that $(\mathbf{\dot C},\dot \rho)$ can be fully determined by $\mathcal{A}'[\mathbf{C}_0,1](\mathbf{\dot C},\dot \rho)$. We mention here that a similar linearized problem in frequency domain has been considered in \cite{yang2019unique}, where a transversely isotropic perturbation is proved to be unique. In this paper, we only use incident waves with finitely many incoming angles $\theta$ and polarization directions $\alpha$. Let $\mathbf{e}_i$ denote the unit vector parallel to the positive $x_i$-axis.
The main result can be stated as follows.
\begin{theorem} \label{maintheorem}
    If $\mathcal{A}'[\mathbf{C}_0,1](\mathbf{\dot{C}},\dot{\rho})(x,t;\mathbf{e}_i,\mathbf{e}_j)=0$ for $(x,t)\in \Omega^c\times(-\infty,T]$ and $i,j=1,2,3$, then $(\mathbf{\dot{C}},\dot{\rho})=0$.
\end{theorem}
We mention here that the 
linearized inverse boundary value problem (at an isotropic reference, not necessarily homogeneous) can be reduced to certain tensor tomography problems of fourth order tensors \cite{de2021generic}. To be more precise, using $P$-waves we end up with the longitudinal ray transform, from which one can determine $5$ parameters; using $S$-waves we end up with the mixed ray transform, from which $9$ parameters can be determined. But the most general anisotropic elastic tensor has $21$ independent parameters. So our result determines a total number of $22$ parameters in $(\mathbf{\dot C},\dot\rho)$. One surprising fact about this result is there is no gauge freedom in the determination, while for the anisotropic acoustic wave equation or Maxwell's equations certain gauge freedoms exist.
We also refer to \cite{Joonas2025gauge} for a detailed study on the gauge freedoms of elastic parameters in a Riemannian geometric setting (instead of the Euclidean setting considered in this article).
In this study, we will adopt the strategy recently used for fixed angle inverse scattering problems \cite{salo2020fixed,rakesh2020fixed,ma2023fixed,ms2020fixed,oksanen2024rigidity1,oksanen2024rigidity2}, which is intrinsically different from the geometrical approaches.

\medskip

For the case where $\mathbf{\dot C}=\dot{\lambda}\delta_{ij}\delta_{k\ell}+\dot{\mu}(\delta_{ik}\delta_{j\ell}+\delta_{i\ell}\delta_{jk})$ is isotropic, we will also establish a Lipschitz stability estimate for the inverse of $\mathcal{A}'[\mathbf{C}_0,1]$, which would give a rigidity result of an isotropic perturbation using Stefanov-Uhlmann's argument \cite{SU12}. This scheme has been established for an acoustic wave equation in \cite{ma2023fixed}. Now we assume $\mathbf{C}=\lambda\delta_{ij}\delta_{k\ell}+\mu(\delta_{ik}\delta_{j\ell}+\delta_{i\ell}\delta_{jk})$ where $\lambda-\lambda_0,\mu-\mu_0,\rho-1\in C_c^\infty(\Omega)$. Assume that
\[
\rho>0, \quad\mu>0,\quad 3\lambda+2\mu>0.
\]
Then the elastic wave equation \eqref{fullelasticWE} reduces to
\begin{equation*}
    \rho \partial_t^2 u- \nabla \cdot (\lambda (\nabla\cdot u) I + (\nabla u+\nabla^T u)) = 0 \quad \text{in } \mathbb{R}^3 \times \mathbb{R}.
\end{equation*}
The equation \eqref{scatteredUeq} satisfied by $U$ can be written as
\[
\begin{split}
 \rho\partial_t^2U-\nabla\cdot(\lambda (\nabla\cdot U)I+(\nabla U+\nabla^T U))&=0,\quad\,\,\,\text{in }\mathbb{R}^3\times\mathbb{R},\\
U&=u_0,\quad\text{for }t\ll 0.
\end{split}
\]
We write the operator $\mathcal{A}$ as
\[
\mathcal{A} \colon (\lambda,\mu,\rho) \mapsto \big\{ U(\cdot,\cdot;\theta,\alpha) \vert_{\Omega^c\times(-\infty,T]} \big\vert \theta \in \mathbb{S}^2,\alpha = \theta \text{ or } \alpha \perp \theta \big\}.
\]
Here we only need to use one incident angle $\theta$ with two polarization directions: $\theta$ for the $P$-wave and $\alpha\perp\theta$ for the $S$-wave.
The main result can be summarized in the following theorem.
\begin{theorem} \label{thm:2-EL25}
   Fix $\theta,\alpha\in\mathbb{S}^2$ with $\alpha\perp\theta$. There exist constants $C > 0$, $s_0 > 3/2$, $k \geq 0$ being an integer, $s \in (-s_0, s_0)$, $\gamma \in (0,1)$ such that for $(\lambda, \mu,\rho)$ sufficiently close to $(\lambda_0,\mu_0,1)$, there holds
    \begin{align*}
        C \| (\lambda, \mu,\rho) - ( \lambda_0, \mu_0,1) \|_{H^{s_0}(\Omega)}
        & \leq \| \mathcal{A}(\lambda,\mu,\rho)(\cdot,\cdot;\theta,\theta) - \mathcal{A}(\lambda_0,\mu_0,1)(\cdot,\cdot;\theta,\theta) \|_{H^{-k}((-\infty,T]; H^{s+2}(\Omega^c))}^\gamma \\
        & + \| \mathcal{A}(\lambda,\mu,\rho)(\cdot,\cdot;\theta,\alpha) - \mathcal{A}(\lambda_0,\mu_0,1)(\cdot,\cdot;\theta,\alpha) \|_{H^{-k}((-\infty,T]; H^{s+2}(\Omega^c))}^\gamma.
    \end{align*}
\end{theorem}

The above result implies that if $\mathcal{A}(\lambda, \mu,\rho)=\mathcal{A}(\lambda_0, \mu_0,1)$ and $(\lambda, \mu,\rho)$ is sufficiently close to constant $(\lambda_0,\mu_0,1)$, then $(\lambda, \mu,\rho)=(\lambda_0,\mu_0,1)$.
We mention that a similar rigidity result was recently established in \cite{ilmavirta2026rigidity} using the Dirichlet-to-Neumann (DtN) map without the closeness assumption. However, DtN map encodes scattering data with all incident angles, while we have only used a single angle.\\

The rest of the paper is organized as follows. In Section \ref{secPP}, we prove Theorem \ref{maintheorem} for the determination of a fully anisotropic perturbation in the linearized problem. In Section \ref{Rigidity}, we revisit the linearized problem with an isotropic perturbation and single-angle data. With Stefanov-Uhlmann's scheme, we also prove  Theorem \ref{thm:2-EL25} for the nonlinear map.

\section{Analysis of scattered waves and injectivity of the linearized map}

In this section we prove Theorem \ref{maintheorem}. We will carefully analyze the single scattered waves $\dot{u}$ generated by incident wave $u_0$.\\

\subsection{Send $P$-wave, observe $P$-wave}\label{secPP}
First, take a $P$-wave solution
\[
u_0=\theta H_0(c_pt-\theta\cdot x),
\]
as the incidence wave. 
We now have
\[
\partial^2_t\dot{u}_i-(\lambda_0+\mu_0)\nabla_i\nabla\cdot \dot{u}-\mu_0\Delta \dot{u}_i=-\partial_j\dot c_{ijk\ell}\theta_k\theta_\ell\delta(c_pt-\theta\cdot x)+\dot c_{ijk\ell}\theta_j\theta_k\theta_\ell\delta'-c_p^2\dot \rho\theta_i\delta'.
\]
Here $\delta$ is the delta distribution. The solution $\dot{u}$ satisfies the above equation in the distributional sense. We first use the scattered $P$-wave information contained in $\dot{u}\vert_{\Omega^c\times(-\infty,T]}$. For this, we take divergence of both sides of the above equation, 
\begin{equation} \label{PPlinearizeEWE}
    \square_p (\nabla\cdot \dot u)
    = -\partial_i \partial_j \dot c_{ijk\ell} \theta_k \theta_\ell \delta + 2\partial_i \dot c_{ijk\ell} \theta_j \theta_k \theta_\ell \delta'-\dot c_{ijk\ell} \theta_i \theta_j \theta_k \theta_\ell \delta''-c_p^2\theta\cdot\nabla\dot \rho\delta'+c_p^2\dot \rho\delta'',
\end{equation}
where $\square_p := \partial^2_t - \Delta_p$ and $\Delta_p := (\lambda_0 + 2\mu_0) \Delta$.
Denote $H_j(s)=s^jH_0(s)$ for $j=1,2,\cdots$.
We will show that $\nabla \cdot \dot{u}$ can be written as the progressive wave expansion
\begin{equation} \label{PPansatz}
\nabla\cdot\dot{u}=w_2(x)\delta'+w_1(x)\delta+w_0(x) H_0+w_{-1}(x) H_1+w'(x,t),
\end{equation}
where $w_2,w_1,w_{0},w_{-1}\in C^\infty(\mathbb{R}^3)$. 
Inserting the ansatz on the right hand side of \eqref{PPansatz} into the equation \eqref{PPlinearizeEWE}, we end up with
\[
\begin{split}
&-c_p^2[\Delta w_2\delta'-2\theta\cdot\nabla w_2\delta'']\\
&-c_p^2[\Delta w_1\delta-2\theta\cdot\nabla w_1\delta']\\
&-c_p^2[\Delta w_{0}H_0-2\theta\cdot\nabla w_{0}\delta]\\
&-c_p^2[\Delta w_{-1}H_1-2\theta\cdot\nabla w_{-1}H_0]\\
&+\square_pw'\\
=&-\partial_i\partial_j\dot c_{ijk\ell}\theta_k\theta_\ell\delta+2\partial_i\dot c_{ijk\ell}\theta_j\theta_k\theta_\ell\delta'-\dot c_{ijk\ell}\theta_i\theta_j\theta_k\theta_\ell\delta''-c_p^2\theta\cdot\nabla\dot \rho\delta'+c_p^2\dot \rho\delta''.
\end{split}
\]
In this subsection the argument for $\delta',\delta,H_j$ is always $c_pt-\theta\cdot x$ and we occasionally omit it.
We then choose $w_2,w_1,w_{0},w_{-1}$ and $w'$ such that
\begin{alignat}{3}
    \partial^2_tw'-c_p^2\Delta w'-c_p^2\Delta w_{-1}H_1 &=0, \quad && \text{in } \mathbb{R}^3\times\mathbb{R}, \label{PPidentity1} \\
    2c_p^2\theta\cdot\nabla w_{-1}-c_p^2\Delta w_{0}&=0,\quad&&\text{in }\mathbb{R}^3,\\
    2c_p^2\theta\cdot\nabla w_{0}-c_p^2\Delta w_1&=-\partial_i\partial_j\dot c_{ijk\ell}\theta_k\theta_\ell,\quad&&\text{in }\mathbb{R}^3, \label{PPidentity2} \\
    2c_p^2\theta\cdot\nabla w_1-c_p^2\Delta w_2&=2\partial_i\dot c_{ijk\ell}\theta_j\theta_k\theta_\ell-c_p^2\theta\cdot\nabla\dot \rho,\quad&&\text{in }\mathbb{R}^3, \label{PPidentity3} \\
    2c_p^2\theta\cdot\nabla w_2&=-\dot c_{ijk\ell}\theta_i\theta_j\theta_k\theta_\ell+c_p^2\dot \rho,\quad&&\text{in }\mathbb{R}^3. \label{PPidentity4}
\end{alignat}
For the construction, we can take
\[
w_2(x)=\frac{1}{2c_p^2}\int_0^\infty -\dot c_{ijk\ell}(x-s\theta)\theta_j\theta_k\theta_\ell+c_p^2\dot \rho(x-s\theta)\mathrm{d}s,
\]
\[
w_1(x)=\frac{1}{2c_p^2}\int_0^\infty c_p^2\Delta w_2(x-s\theta)+2\partial_i\dot c_{ijk\ell}(x-s\theta)\theta_j\theta_k\theta_\ell-c_p^2\theta\cdot\nabla\dot \rho(x-s\theta)\mathrm{d}s,
\]
\[
w_{0}(x)=\frac{1}{2c_p^2}\int_0^\infty c_p^2\Delta w_1(x-s\theta)-\partial_i\partial_j\dot c_{ijk\ell}(x-s\theta)\theta_k\theta_\ell\mathrm{d}s,
\]
and
\[
w_{-1}(x)=\frac{1}{2}\int_0^\infty\Delta w_{0}(x-s\theta)\mathrm{d}s.
\]
By above construction $w_2,w_1,w_{0},w_{-1}$ all vanish in $\{x\cdot\theta<-1\}$, then $w'$ has to satisfy \eqref{PPidentity1} such that $w'=0$ for $t\ll 0$. By standard hyperbolic theory (cf., for example, \cite[Chapter 7]{EvanPDEs}), there exists a unique solution $w'\in H^2_{\mathrm{loc}}$, and it vanishes in $\{c_pt<\theta\cdot x\}$. This justifies our ansatz \eqref{PPansatz}. We also refer the proof of \cite[Proposition 3.1]{ma2023fixed} for more details.

\medskip

We can now write
\begin{equation} \label{PPansatz1}
\nabla\cdot\dot{u}=w_2(x)\delta'+w_1(x)\delta+v'(x,t) H_0,
\end{equation}
by taking 
\[
v'(x,t)=w_{0}(x)+w_{-1}(x)(c_pt-\theta\cdot x)+w'(x,t)\in H^2_{\mathrm{loc}}.
\]
Now let us assume that 
\[
\mathcal{A}'[\mathbf{C}_0,1](\mathbf{\dot C},\dot \rho)=0,
\]
or equivalently $\dot{u}=0$ in $\Omega^c\times(-\infty,T]$. Then $\nabla\cdot\dot{u}=0$ in $\Omega^c\times(-\infty,T]$. By the ansatz \eqref{PPansatz1} we conclude that 
\[
v'=0\quad\text{in } \{c_pt>\theta\cdot x\}\cap(\Omega^c\times(-\infty,T]).
\]
Also, since $\nabla\cdot\dot{u}=v'$ in $\{c_pt>\theta\cdot x\}$, we have
\[
\square_p v'=0,\quad\text{in }\{c_pt>\theta\cdot x\}.
\]
We recall the unique continuation principle, see e.g.~\cite[Theorem 3.16]{katchalov2001}.

\begin{proposition} \label{th_ucp}
Let $\omega \subset \R^n$ be open, and $T, c \in (0,+\infty)$, $t_0 \in \R$ all be constants. Define 
\begin{equation*}
K = K(\omega, T, t_0) = \{ (x, t) \in \R^{n + 1} \mid d(x, \omega) / c < T - |t - t_0| \},
\end{equation*}
where $d$ is the distance function on $\R^n$. 
Let $u \in H^1(K)$ satisfy $\Box_c u = 0$ in $K$ and $u = 0$ in $\omega \times (t_0 - T, t_0 + T)$.
Then $u = 0$ in $K$.
\end{proposition}

\begin{lemma} \label{lem_ucp}
Let $s, a \in \R$.
Suppose that $u \in H^1(\R^{n + 1})$ satisfies $\Box_c u = 0$ in the set
\begin{equation*}
 R= \{ (x, t) \in \R^{n + 1} \mid \theta \cdot x < \min(c(t - s), a) \}    
\end{equation*}
and that $u = 0$ in $\Omega^c \times (-\infty,T)$ for $T$ sufficiently large.
Then $u = 0$ in $R$.
\end{lemma}
\begin{proof}
After a rotation in space and translation in time we may assume that $\theta = (1,0,0)$ and $s = 0$. Let $x_0 = (x_0^1, x_0^2, x_0^3) \in \Omega$ satisfy $x_0^1 < a$ and let $t_0 > x_0^1 / c$.
Choose large enough $r > 0$ so that $y := x_0 - r \theta$ has a neighborhood $\omega$ that is contained in $\Omega^c$. Then the closed diamond $\overline{K(\{y\}, r/c, t_0)}$
contains $(x_0, t_0)$ and is contained in $R$. 
For a small enough neighborhood $\omega \subset \Omega^c$ of $y$, the open diamond $K(\omega, r/c, t_0)$
also contains $(x_0, t_0)$ and is contained in $R$. If $T$ is taken sufficiently large, we have $u=0$ in $\omega\times (t_0-r/c,t_0+r/c)$. It follows from Proposition \ref{th_ucp} that $v' = 0$ near $(x_0, t_0)$.

To see $\overline{K(\{y\}, r/c, t_0)}$ is indeed contained in $R$, we argue as follows. 
Let $(x,t) \in \overline{K(\{y\}, r/c, t_0)}$. It follows from
\begin{equation*}
|x^1 - x^1_0 + r|/c \le r/c
\end{equation*}
that $x^1 \le x^1_0$. In particular $x^1 < a$.
To get a contradiction, suppose that $c t \le x^1$.
We have
\begin{equation*}
t \le x^1/c \le x_0^1/c < t_0, \quad |t_0 - t| = t_0 - t > |x_0^1 - x^1| / c
\end{equation*}
and
\begin{equation*}
|x^1 - x^1_0 + r|/c + |x_0^1 - x^1| / c
< 
|x^1 - x^1_0 + r|/c + |t_0 - t| \le r/c.
\end{equation*}
Finally, we arrive to the contradiction
\begin{equation*}
r 
= 
r - |x^1 - x^1_0| + |x_0^1 - x^1|
\le 
|x^1 - x^1_0 + r| + |x_0^1 - x^1|
< r.
\end{equation*}
The proof is done.
\end{proof}

Taking $s = 0$ and large enough $a > 0$, it follows from Lemma \ref{lem_ucp} that 
that $v'=0$ in $\{c_pt>\theta\cdot x\}$ and then also on $\{c_pt=\theta\cdot x\}$ by continuity.

Notice that $v'(x,t)=w_{0}(x)$ on $\{c_pt=\theta\cdot x\}$, thus now we have $w_{0}=0$ in $\mathbb{R}^3$. Then the identity \eqref{PPidentity2} reduces to
\[
-c_p^2\Delta w_1=-\partial_i\partial_j\dot c_{ijk\ell}\theta_k\theta_\ell\quad \text{in }\mathbb{R}^3.
\]
Together with \eqref{PPidentity3} and \eqref{PPidentity4}, we end up with the equation
\[
4(\theta\cdot\nabla)^2\partial_i\partial_j\dot c_{ijk\ell}\theta_k\theta_\ell+\Delta^2\dot c_{ijk\ell}\theta_i\theta_j\theta_k\theta_\ell-c_p^2\Delta^2\dot \rho=4(\theta\cdot\nabla)\Delta\partial_i\dot c_{ijk\ell}\theta_j\theta_k\theta_\ell-2c_p^2(\theta\cdot\nabla)^2\Delta\dot \rho.
\]
In particular, if we take $\theta=\mathbf{e}_1$, the above equation reads
\begin{equation} \label{PP1}
4\partial_1^2\partial_i\partial_j\dot c_{ij11}+\Delta^2\dot c_{1111}-c_p^2\Delta^2\dot \rho=4\partial_1\Delta\partial_i\dot c_{i111}-2c_p^2\partial_1^2\Delta\dot \rho.
\end{equation}

\subsection{Send $S$-wave, observe $P$-wave}
Now take an $S$-wave solution
\[
u_0=\alpha H_0(c_st-\theta\cdot x),\quad \alpha\perp\theta,
\]
we have
\[
\partial^2_t\dot{u}_i-(\lambda_0+\mu_0)\nabla_i\nabla\cdot \dot{u}-\mu_0\Delta \dot{u}_i=-\partial_j\dot c_{ijk\ell}\theta_k\alpha_\ell\delta(c_st-\theta\cdot x)+\dot c_{ijk\ell}\theta_j\theta_k\alpha_\ell\delta'-c_s^2\dot \rho\alpha_i\delta'.
\]
In this subsection the argument for $\delta'',\delta',\delta,H_j$ is always $c_st-\theta\cdot x$ and we occasionally omit it.
Taking divergence of both sides, we have
\begin{equation} \label{SPlinearizeEWE}
    [\partial^2_t-c_p^2\Delta] (\nabla\cdot \dot{u}) = -\partial_i \partial_j \dot c_{ijk\ell} \theta_k \alpha_\ell \delta (c_st-\theta\cdot x) + 2\partial_i \dot c_{ijk\ell} \theta_j \theta_k \alpha_\ell \delta' - \dot c_{ijk\ell} \theta_i \theta_j \theta_k \alpha_\ell \delta'' - c_s^2 \alpha_i \partial_i \dot \rho \delta'.
\end{equation}
We use the ansatz
\begin{equation} \label{SPansatz}
    \nabla\cdot\dot{u}=w_{1}(x)\delta(c_st-\theta\cdot x)+w_{0}(x) H_0+w_{-1}(x) H_1+w_{-2}(x) H_2+w'(x,t),
\end{equation}
where $w_1,w_{0},w_{-1},w_{-2}\in C^\infty(\mathbb{R}^3)$ are all smooth functions. Inserting the above ansatz into \eqref{SPlinearizeEWE}, one obtains
\[
\begin{split}
&\quad c_s^2w_1\delta''-c_p^2\left[\Delta w_1\delta-2\theta\cdot\nabla w_1\delta'+w_1(x)\delta''\right]\\
&+c_s^2w_{0}\delta'-c_p^2\left[\Delta w_{0}H_0-2\theta\cdot\nabla w_{0}\delta+w_{0}\delta'\right]\\
&+c_s^2w_{-1}\delta-c_p^2\left[\Delta w_{-1}H_1-2\theta\cdot\nabla w_{-1}H_0+w_{-1}\delta\right]\\
&+c_s^2w_{-2}H_0-c_p^2\left[\Delta w_{-2}H_2-2\theta\cdot\nabla w_{-2}H_1+w_{-2}H_0\right]\\
&+\square_pw'\\
=&-\partial_i\partial_j\dot c_{ijk\ell}\theta_k\alpha_\ell\delta(c_st-\theta\cdot x)+2\partial_i\dot c_{ijk\ell}\theta_j\theta_k\alpha_\ell\delta'-\dot c_{ijk\ell}\theta_i\theta_j\theta_k\alpha_\ell\delta''-c_s^2\alpha_i\partial_i\dot \rho\delta'.
\end{split}
\]
We can take $w_1,w_{0},w_{-1},w_{-2}$ and $w'$ such that
\begin{alignat}{3}
\square_p w'-c_p^2\Delta w_{-1}H_1-c_p^2\Delta w_{-2}H_2+2c_p^2\theta\cdot\nabla w_{-2}H_1&=0,\quad&&\text{in }\mathbb{R}^3\times\mathbb{R}, \label{SPidentity1}\\
(c_s^2-c_p^2)w_{-2}+2c_p^2\theta\cdot\nabla w_{-1}-c_p^2\Delta w_{0}&=0,\quad&&\text{in }\mathbb{R}^3,\\
(c_s^2-c_p^2)w_{-1}+2c_p^2\theta\cdot\nabla w_{0}-c_p^2\Delta w_1&=-\partial_i\partial_j\dot c_{ijk\ell}\theta_k\alpha_\ell,\quad&&\text{in }\mathbb{R}^3, \label{SPidentity2}\\
(c_s^2-c_p^2)w_{0}+2c_p^2\theta\cdot\nabla w_1&=2\partial_i\dot c_{ijk\ell}\theta_j\theta_k\alpha_\ell-c_s^2\alpha_i\partial_i\dot \rho,\quad&&\text{in }\mathbb{R}^3, \label{SPidentity3}\\
(c_s^2-c_p^2) w_1 &= -\dot c_{ijk\ell}\theta_i\theta_j\theta_k\alpha_\ell, \quad && \text{in }\mathbb{R}^3. \label{SPidentity4}
\end{alignat}
Since $c_p\neq c_s$, we just successively take
\[
w_1(x)=-\frac{1}{(c_s^2-c_p^2)}\dot c_{ijk\ell}(x)\theta_i\theta_j\theta_k\alpha_\ell,
\]
\[
w_{0}(x)=\frac{1}{c_s^2-c_p^2}(-2c_p^2\theta\cdot\nabla w_1(x)+2\partial_i\dot c_{ijk\ell}(x)\theta_j\theta_k\alpha_\ell-c_s^2\alpha_i\partial_i\dot \rho),
\]
\[
w_{-1}(x)=\frac{1}{c_s^2-c_p^2}\left(-2c_p^2\theta\cdot\nabla w_{0}+c_p^2\Delta w_1-\partial_i\partial_j\dot c_{ijk\ell}\theta_k\alpha_\ell\right),
\]
and 
\[
w_{-2}(x)=\frac{1}{c_s^2-c_p^2}\left(-2c_p^2\theta\cdot\nabla w_{-1}+c_p^2\Delta w_{0}\right).
\]
Note that by the construction above, $w_{j}$ for $j \in \{1,0,-1,-2\}$ are all smooth functions and compactly supported in $\Omega$. Then we let $w'$ satisfy the equation \eqref{SPidentity1} with $w'=0,t\ll 0$, which has a unique solution in $H^2_{\mathrm{loc}}$. This justifies our ansatz \eqref{SPansatz}.

Now we can write
\begin{equation} \label{SPansatz1}
    \nabla\cdot\dot{u}=w_1(x)\delta(c_st-\theta\cdot x)+v'(x,t) H_0(c_st-\theta\cdot x)+w'(x,t)
\end{equation}
by taking 
\begin{equation} \label{SPvprime}
    v'(x,t)=w_{0}(x)+(c_st-\theta\cdot x)w_{-1}(x)+(c_st-\theta\cdot x)^2w_{-2}(x).
\end{equation}
Then $v'$ and $w'$ solve
\[
\square_p(v'+w')=0,\quad \text{in }\{c_st>\theta\cdot x\},
\]
and
\[
\square_pw'=0,\quad \text{in }\{c_st<\theta\cdot x\}.
\]

\begin{figure}[htbp]
    \centering
    \includegraphics[width=0.6\textwidth]{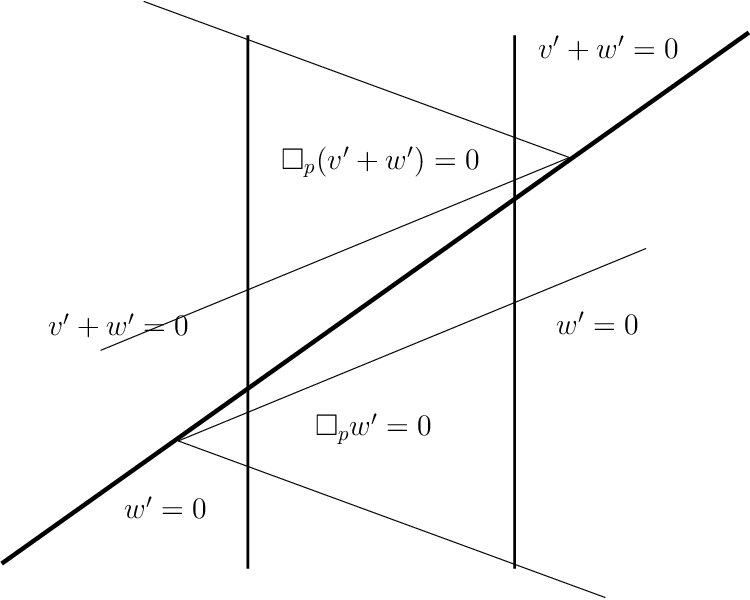}
    \caption{Unique continuation for $S$-$P$ scattering.}
    \label{fig:1}
\end{figure}
Notice that $\nabla\cdot\dot{u}=v'+w'$ in $\{c_st>\theta\cdot x\}$ and $\nabla\cdot\dot{u}=w'$ in $\{c_st<\theta\cdot x\}$. 
For the inverse problem, we start with the fact $\nabla\cdot\dot{u}=$ in $\Omega^c\times(-\infty,T]$.
By unique continuation for the Cauchy problem
\[
\begin{split}
\square_pw'&=0,\quad \text{in }\{c_st<\theta\cdot x\},\\
w'&=0,\quad \text{in }\{c_st<\theta\cdot x\}\cap (\Omega^c\times(-\infty,T]),
\end{split}
\]
we can conclude $w'\equiv 0$ in $\{c_st<\theta\cdot x\}$. This follows from a time-reversed version of the argument that we will give next.
By unique continuation for the Cauchy problem
\[
\begin{split}
\square_p(v'+w')&=0,\quad \text{in }\{c_st>\theta\cdot x\},\\
v'+w'&=0,\quad \text{in }\{c_st>\theta\cdot x\}\cap (\Omega^c\times(-\infty,T]),
\end{split}
\]
we have $v'+w'\equiv 0$ in $\{c_st>\theta\cdot x\}$.
See Fig.~\ref{fig:1} for an illustration.
Indeed, suppose that $(x_0, t_0)$ satisfies $c_s t_0 > \theta \cdot x_0$.
We apply Lemma~\ref{lem_ucp} with
\begin{equation*}
a = \theta \cdot x_0, \quad s = \frac{(c_p - c_s) a}{c_p c_s}.
\end{equation*}
With this choice of $(a,s)$, $\min(c_p(t - s), a) \le c_s t$ for all $t$ since 
either $a \le c_s t$ or $c_s t < a$, and in the latter case
\begin{equation*}
c_p (t - s) = c_p t - \frac{(c_p - c_s) a}{c_s}
< c_p t - (c_p - c_s) t = c_s t.
\end{equation*}
In particular, the set $R$ in Lemma~\ref{lem_ucp} is contained in the region $c_st>\theta\cdot x$, 
and we obtain $v'+w' = 0$ in $R$.
Moreover, $(x_0, t_0) \in \overline{R}$ since 
\begin{equation*}
c_p (t_0 - s) = c_p t_0 - \frac{(c_p - c_s) a}{c_s}
> c_p t_0 - (c_p - c_s) t_0 = c_s t_0 > \theta \cdot x_0.
\end{equation*}
By perturbing $(x_0, t_0)$ slightly, we conclude that $v'+w' = 0$ near $(x_0, t_0)$.

Notice that $v'$ is smooth and $w'$ is in $H^2_{\mathrm{loc}}$. Taking the traces on $\{c_st=\theta\cdot x\}$, we have $v'+w'=w'=0$ and $(\partial_t+c_p\theta\cdot \nabla)(v'+w')=(\partial_t+c_p\theta\cdot \nabla)w'=0$ on $\{c_st=\theta\cdot x\}$. This means that $v'=(\partial_t+c_p\theta\cdot \nabla)v'=0$ on $\{c_st=\theta\cdot x\}$ and thus $w_{0}(x)=w_{-1}(x)=0$ by \eqref{SPvprime}.

\medskip

Then the identities \eqref{SPidentity2} and \eqref{SPidentity3} become
\begin{alignat*}{3}
-c_p^2\Delta w_1&=-\partial_i\partial_j\dot c_{ijk\ell}\theta_k\alpha_\ell,\quad&\text{in }\mathbb{R}^3,\\
2c_p^2\theta\cdot\nabla w_1&=2\partial_i\dot c_{ijk\ell}\theta_j\theta_k\alpha_\ell-c_s^2\alpha_i\partial_i\dot \rho,\quad&\text{in }\mathbb{R}^3.
\end{alignat*}
Together with the equation \eqref{SPidentity4}, we obtain
\begin{alignat*}{2}
2\frac{c_p^2}{c_p^2-c_s^2}\theta\cdot\nabla \dot c_{ijk\ell}\theta_i\theta_j\theta_k\alpha_\ell&=2\partial_i\dot c_{ijk\ell}\theta_j\theta_k\alpha_\ell-c_s^2\alpha_i\partial_i\dot \rho,\quad&\text{in }\mathbb{R}^3,\\
\frac{c_p^2}{c_p^2-c_s^2}\Delta \dot c_{ijk\ell}\theta_i\theta_j\theta_k\alpha_\ell&=\partial_i\partial_j\dot c_{ijk\ell}\theta_k\alpha_\ell,\quad&\text{in }\mathbb{R}^3.
\end{alignat*}
In particular, if we take $\theta=\mathbf{e}_1,\alpha=\mathbf{e}_2$ in above equations, we have
\begin{alignat}{2}
    2\frac{c_p^2}{c_p^2-c_s^2}\partial_1\dot c_{1112}&=2\partial_i\dot c_{i112}-c_s^2\partial_2\dot \rho,\quad&\text{in }\mathbb{R}^3 \label{SP121},\\
    \frac{c_p^2}{c_p^2-c_s^2}\Delta \dot c_{1112}&=\partial_i\partial_j\dot c_{ij12},\quad&\text{in }\mathbb{R}^3. \label{SP122}
\end{alignat}
Taking $\theta=\mathbf{e}_2,\alpha=\mathbf{e}_1$, we have
\begin{alignat}{2}
    2\frac{c_p^2}{c_p^2-c_s^2}\partial_2\dot c_{2221}&=2\partial_i\dot c_{i221}-c_s^2\partial_1\dot \rho,\quad&\text{in }\mathbb{R}^3 \label{SP211},\\
    \frac{c_p^2}{c_p^2-c_s^2}\Delta \dot c_{2221}&=\partial_i\partial_j\dot c_{ij21},\quad&\text{in }\mathbb{R}^3. \label{SP212}
\end{alignat}
The equations \eqref{SP122} and \eqref{SP212} immediately imply that $\Delta(\dot c_{1112}-\dot c_{2221})=0$ in $\mathbb{R}^3$, and since $\mathbf{\dot C}$ is compactly supported in $\Omega$,
\[
\dot c_{1112}=\dot c_{2221},\quad\text{in }\mathbb{R}^3.
\]
Now we write \eqref{SP121} and \eqref{SP211} as
\begin{equation} \label{identity1}
    2\frac{c_s^2}{c_p^2-c_s^2}\partial_1\dot c_{1112}=2\partial_2\dot c_{2112}+2\partial_3\dot c_{3112}-c_s^2\partial_2\dot \rho,
\end{equation}
\begin{equation} \label{identity2}
    2\frac{c_s^2}{c_p^2-c_s^2}\partial_2\dot c_{2221}=2\partial_1\dot c_{1221}+2\partial_3\dot c_{3221}-c_s^2\partial_1\dot \rho.
\end{equation}

\subsection{Send $P$-wave, observe $S$-wave}
In this step, we take a $P$-wave solution
\[
u_0=\theta H_0(c_pt-\theta\cdot x),
\]
and observed the scattered $S$-waves.
Now we have
\[
\partial^2_t\dot{u}_i-(\lambda_0+\mu_0)\nabla_i\nabla\cdot \dot{u}-\mu_0\Delta \dot{u}_i=-\partial_j\dot c_{ijk\ell}\theta_k\theta_\ell\delta+\dot c_{ijk\ell}\theta_j\theta_k\theta_\ell\delta'-c_p^2\dot \rho\theta_i\delta'.
\]
We introduce the notation $e_{ijk}$ that is skew-symmetric and $e_{123}=1$.
Taking curl of both sides
\[
\begin{split}
\square_{s}(\nabla\times u)_i=&e_{ipq}\partial_q(-\partial_j\dot c_{pjk\ell}\theta_k\theta_\ell\delta({c_p}t-\theta\cdot x)+\dot c_{pjk\ell}\theta_j\theta_k\theta_\ell\delta')-c_p^2e_{ipq}\partial_q(\dot \rho\theta_p\delta')\\
=&-e_{ipq}\dot c_{pjk\ell}\theta_j\theta_k\theta_\ell\theta_q\delta''+e_{ipq}\partial_q\dot c_{pjk\ell}\theta_j\theta_k\theta_\ell\delta'+e_{ipq}\partial_j\dot c_{pjk\ell}\theta_k\theta_\ell\theta_q\delta'-e_{ipq}\partial_q\partial_j\dot c_{pjk\ell}\theta_k\theta_\ell\delta.\\
&-c_p^2e_{ipq}\partial_q\dot \rho\theta_p\delta'+c_p^2e_{ipq}\dot \rho\theta_p\theta_q\delta''.
\end{split}
\]
In this subsection the argument for $\delta'',\delta',\delta,H_j$ is always $c_pt-\theta\cdot x$ and we occasionally omit it.
We use the ansatz
\begin{equation} \label{PSansatz}
\nabla\times\dot{u}=w_1(x)\delta(c_pt-\theta\cdot x)+w_{0}(x) H_0+w_{-1}(x) H_1+w_{-2}(x) H_2+w'(x,t),
\end{equation}
where $v,w$ are all smooth functions. Insert into the equation above,
\[
\begin{split}
&c_p^2w_{1i}\delta''-c_s^2\left[\Delta w_{1i}\delta-2\theta\cdot\nabla w_{1i}\delta'+w_{1i}\delta''\right]\\
&+c_p^2w_{0i}\delta'-c_s^2\left[\Delta w_{0i}H_0-2\theta\cdot\nabla w_{0i}\delta+w_{0i}\delta'\right]\\
&+c_p^2w_{-1i}\delta-c_s^2\left[\Delta w_{-1i}H_1-2\theta\cdot\nabla w_{-1i}H_0+w_{-1i}\delta\right]\\
&+c_p^2w_{-2i}H_0-c_s^2\left[\Delta w_{-2i}H_2-2\theta\cdot\nabla w_{-2i}H_1+w_{-2i}H_0\right]\\
&+\square_sw'_i\\
=&e_{ipq}\dot c_{pjk\ell}\theta_j\theta_k\theta_\ell\theta_q\delta''+e_{ipq}\partial_q\dot c_{pjk\ell}\theta_j\theta_k\theta_\ell\delta'+e_{ipq}\partial_j\dot c_{pjk\ell}\theta_k\theta_\ell\theta_q\delta'-e_{ipq}\partial_q\partial_j\dot c_{pjk\ell}\theta_k\theta_\ell\delta\\
&-c_p^2e_{ipq}\partial_q\dot \rho\theta_p\delta'.
\end{split}
\]
Now we take $w_1,w_0,w_{-1},w_{-2}$ such that
\begin{alignat}{3}
    (c_p^2-c_s^2)w_{1i}&=-e_{ipq}\dot c_{pjk\ell}\theta_j\theta_k\theta_\ell\theta_q,&\quad\text{in }\mathbb{R}^3, \label{PSidentity1}\\
    (c_p^2-c_s^2)w_{0i}&=-2c_s^2\theta\cdot\nabla w_{1i}+e_{ipq}\partial_q\dot c_{pjk\ell}\theta_j\theta_k\theta_\ell+e_{ipq}\partial_j\dot c_{pjk\ell}\theta_k\theta_\ell\theta_q \label{PSidentity2}\\
    &\quad\quad-c_p^2e_{ipq}\partial_q\dot \rho\theta_p,\quad&\text{in }\mathbb{R}^3,\nonumber\\
    (c_p^2-c_s^2)w_{-1i}&=c_s^2\Delta w_{1i}-2c_s^2\theta\cdot\nabla w_{0i}-e_{ipq}\partial_q\partial_j\dot c_{pjk\ell}\theta_k\theta_\ell,&\quad\text{in }\mathbb{R}^3, \label{PSidentity3}\\
    (c_p^2-c_s^2)w_{-2i}&=c_s^2\Delta w_{0i}-2c_s^2\theta\cdot\nabla w_{-1i},&\quad\text{in }\mathbb{R}^3, \label{PSidentity4}
\end{alignat}
The function $w'$ satisfies the equation
\[
\begin{split}
\square_sw'-c_s^2\Delta w_{-1}H_1-c_s^2\Delta w_{-2}H_2+2c_s^2\theta\cdot\nabla w_{-2}H_1&=0,\quad \text{in }\mathbb{R}^3\times\mathbb{R},\\
w'&=0,\quad t\ll 0.
\end{split}
\]
Then above equation has a unique solution $w'\in H^2_{\mathrm{loc}}$ such that $w'=0$ in $\{c_pt<\theta\cdot x\}$. This justifies our ansatz \eqref{PSansatz}.

Therefore we can write
\[
\nabla\times\dot{u}=w_1(x)\delta(c_pt-\theta\cdot x)+v'(x,t) H_0,
\]
by taking 
\[
v'(x,t)=w_{0}(x)+w_{-1}(x)(c_pt-\theta\cdot x)+w_{-2}(x)(c_pt-\theta\cdot x)^2+w'(x,t),
\]
which is in $H^2_{\mathrm{loc}}$. Then $v'$ solves
\[
\square_sv'=0,\quad \text{in }\{c_pt>\theta\cdot x\}.
\]
\begin{figure}[htbp]
    \centering
    \includegraphics[width=0.6\textwidth]{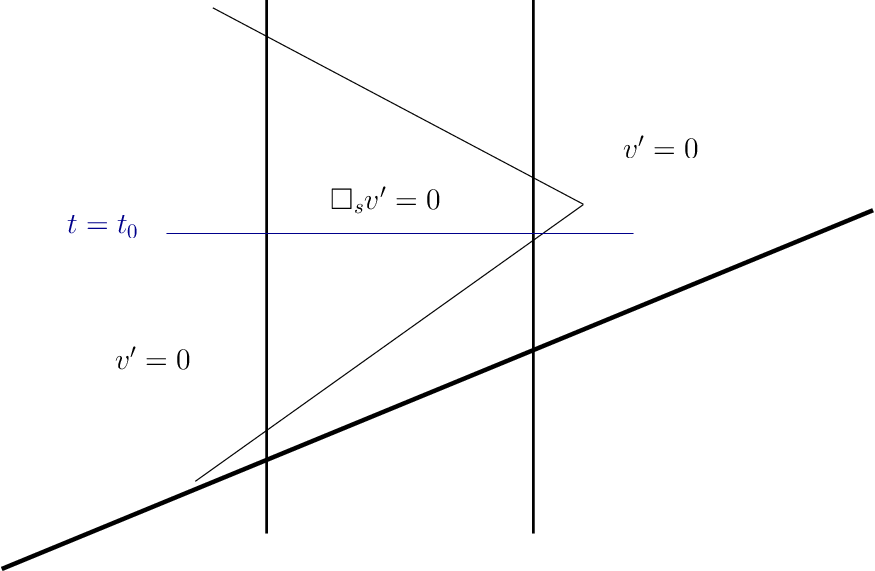}
    \caption{Unique continuation for $P$-$S$ scattering.}
    \label{fig:2}
\end{figure}

For the inverse problem, if $\nabla\times\dot{u}=0$ in $\Omega^c\times(-\infty,T]$, we have
\[
\begin{split}
\square_sv'&=0,\quad \text{in }\{c_pt>\theta\cdot x\},\\
v'&=0,\quad \text{in }\{c_pt>\theta\cdot x\}\cap (\Omega^c\times(-\infty,T]).
\end{split}
\]
We apply Lemma \ref{lem_ucp} with large enough $a, s > 0$. Then $R \cap (\Omega \times \R)$ is in the region $c_pt>\theta\cdot x$ and contains a neighborhood of $\Omega \times \{t_0\}$ for some $t_0 \in \R$. So $v'=0$ in a neighborhood of $\Omega \times \{t_0\}$. We conclude that $v'=0$ in $\{c_pt>\theta\cdot x\}$ by noticing that $v' = 0$ is the unique solution to
\begin{align}
\Box_s v' &= 0,\quad \text{in } \{c_pt>\theta\cdot x\} \cap (\Omega \times (-\infty, t_0)),\\
v'&=0,\quad \text{in }\{c_pt>\theta\cdot x\}\cap (\partial \Omega\times (-\infty, t_0)),\\
v' = \partial_t v' &= 0, \quad \text{in } \Omega\times \{t_0\}.
\end{align}

By continuity, we obtain $v'=(\partial_t+c_s\theta\cdot\nabla)v'=0$ on $\{c_pt=\theta\cdot x\}$. This implies that
\[
w_{0}(x)=w_{-1}(x)=0\quad\text{in }\mathbb{R}^3.
\]
See Fig.~\ref{fig:2} for an illustration.
Then the equations \eqref{PSidentity2} and \eqref{PSidentity3} become
\begin{alignat*}{2}
2c_s^2\theta\cdot\nabla w_{1i}&=e_{ipq}\partial_q\dot c_{pjk\ell}\theta_j\theta_k\theta_\ell+e_{ipq}\partial_j\dot c_{pjk\ell}\theta_k\theta_\ell\theta_q-c_p^2e_{ipq}\partial_q\dot \rho\theta_p,&&\quad\text{in }\mathbb{R}^3\\
c_s^2\Delta w_{1i}&=e_{ipq}\partial_q\partial_j\dot c_{pjk\ell}\theta_k\theta_\ell,&&\quad\text{in }\mathbb{R}^3.
\end{alignat*}
Together with the expression $w_{1i}=-\frac{1}{c_p^2-c_s^2}e_{ipq}\dot c_{pjk\ell}\theta_j\theta_k\theta_\ell\theta_q$, we obtain
\begin{alignat*}{1}
-2\frac{c_s^2}{c_p^2-c_s^2}e_{ipq}\theta\cdot\nabla \dot c_{pjk\ell}\theta_j\theta_k\theta_\ell\theta_q&=e_{ipq}\partial_q\dot c_{pjk\ell}\theta_j\theta_k\theta_\ell+e_{ipq}\partial_j\dot c_{pjk\ell}\theta_k\theta_\ell\theta_q-c_p^2e_{ipq}\partial_q\dot \rho\theta_p,\\
\frac{c_s^2}{c_p^2-c_s^2}e_{ipq}\Delta \dot c_{pjk\ell}\theta_j\theta_k\theta_\ell\theta_q&=-e_{ipq}\partial_q\partial_j\dot c_{pjk\ell}\theta_k\theta_\ell.
\end{alignat*}
Take $\theta=\mathbf{e}_1$ in above identities, we obtain
\begin{alignat}{1} \label{PS112}
    i=1:\quad\quad\quad\quad e_{123}\partial_3\dot c_{2111}+e_{132}\partial_2\dot c_{3111}&=0,\\
    -e_{123}\partial_3\partial_j\dot c_{2j11}-e_{132}\partial_2\partial_j\dot c_{3j11}&=0,
\end{alignat}
\begin{alignat}{1}
    i=2:\quad -2\frac{c_s^2}{c_p^2-c_s^2}e_{231}\partial_1\dot c_{3111}&=e_{213}\partial_3\dot c_{1111}+e_{231}\partial_1\dot c_{3111}+e_{231}\partial_j\dot c_{3j11}-c_p^2e_{213}\partial_3\dot \rho,\\
    \frac{c_s^2}{c_p^2-c_s^2}e_{231}\Delta \dot c_{3111}&=-e_{213}\partial_3\partial_j\dot c_{1j11}-e_{231}\partial_1\partial_j\dot c_{3j11}, \label{PS122}
\end{alignat}
\begin{alignat}{1}
    i=3:\quad -2\frac{c_s^2}{c_p^2-c_s^2}e_{321}\partial_1\dot c_{2111}&=e_{312}\partial_2\dot c_{1111}+e_{321}\partial_1\dot c_{2111}+e_{321}\partial_j\dot c_{2j11}-c_p^2e_{312}\partial_2\dot \rho, \label{PS131}\\
    \frac{c_s^2}{c_p^2-c_s^2}e_{321}\Delta \dot c_{2111}&=-e_{312}\partial_2\partial_j\dot c_{1j11}-e_{321}\partial_1\partial_j\dot c_{2j11}. \label{PS132}
\end{alignat}
Notice that the above six equations are linearly dependent, and two of them are redundant. 

\subsection{Send $S$-wave, observe $S$-wave}\label{sectSS}
Take an $S$-wave solution
\[
u_0=\alpha H_0(c_st-\theta\cdot x),\quad \alpha\perp\theta,
\]
as the incident wave,
we have
\[
\partial^2_t\dot{u}_i-(\lambda_0+\mu_0)\nabla_i\nabla\cdot \dot{u}-\mu_0\Delta \dot{u}_i=-\partial_jc_{ijk\ell}\theta_k\alpha_\ell\delta+c_{ijk\ell}\theta_j\theta_k\alpha_\ell\delta'-c_s^2\rho\alpha_i\delta'.
\]
Taking curl of both sides
\[
\begin{split}
\square_s(\nabla\times u)_i=&-e_{ipq}\dot c_{pjk\ell}\theta_j\theta_k\alpha_\ell\theta_q\delta''+e_{ipq}\partial_q\dot c_{pjk\ell}\theta_j\theta_k\alpha_\ell\delta'+e_{ipq}\partial_j\dot c_{pjk\ell}\theta_k\alpha_\ell\theta_q\delta'-e_{ipq}\partial_q\partial_j\dot c_{pjk\ell}\theta_k\alpha_\ell\delta\\
&-c_s^2e_{ipq}\partial_q\dot \rho\alpha_p\delta'+c_s^2e_{ipq}\dot \rho\alpha_p\theta_q\delta''.
\end{split}
\]
By the same consideration as for $PP$-scattering, we can use the ansatz
\[
\nabla\times\dot{u}=w_2(x)\delta'+w_1(x)\delta+v'(x,t) H_0.
\]
with $v'\in H^2_{\mathrm{loc}}$.
We end up with
\begin{alignat}{3}
    \partial^2_tv'-c_s^2\Delta v'&=0,&&\quad \text{in }\{c_st>\theta\cdot x\},\label{SSidentity1}\\
    2c_s(\partial_tv'+c_s\theta\cdot\nabla v')-c_s^2\Delta w_1&=-e_{ipq}\partial_q\partial_j\dot c_{pjk\ell}\theta_k\alpha_\ell,&&\quad\text{on }\{c_st=\theta\cdot x\},\label{SSidentity2}\\
    2c^2_s\theta\cdot\nabla w_1-c_s^2\Delta w_2&=e_{ipq}\partial_q\dot c_{pjk\ell}\theta_j\theta_k\alpha_\ell+e_{ipq}\partial_j\dot c_{pjk\ell}\theta_k\alpha_\ell\theta_q\label{SSidentity3}\\
    &\quad\quad-c_s^2e_{ipq}\partial_q\dot \rho\alpha_p,&&\quad \text{in }\mathbb{R}^3,\nonumber\\
    2c_s^2\theta\cdot\nabla w_2&=-e_{ipq}\dot c_{pjk\ell}\theta_j\theta_k\alpha_\ell\theta_q+c_s^2e_{ipq}\dot \rho\alpha_p\theta_q,&&\quad \text{in }\mathbb{R}^3.\label{SSidentity4}
\end{alignat}

Similar to the consideration for the case of $PP$-scattering. applying the unique continuation principle to the Cauchy problem
\[
\begin{split}
\square_sv'&=0,\quad \text{in }\{c_st>\theta\cdot x\},\\
v'&=0,\quad \text{in }\{c_st>\theta\cdot x\}\cap (\Omega^c\times(-\infty,T]),
\end{split}
\]
we conclude that $v'=0$ on $\{c_st=\theta\cdot x\}$. Then
\[
c_s^2\Delta w_1=e_{ipq}\partial_q\partial_j\dot c_{pjk\ell}\theta_k\alpha_\ell
\]
by equation \eqref{SSidentity2}. Together with \eqref{SSidentity3} and \eqref{SSidentity4} we end up with the identity
\[
\begin{split}
&4e_{ipq}(\theta\cdot\nabla)^2\partial_q\partial_j\dot c_{pjk\ell}\theta_k\alpha_\ell+e_{ipq}\Delta^2\dot c_{pjk\ell}\theta_j\theta_k\alpha_\ell\theta_q-c_s^2e_{ipq}\Delta^2\dot \rho\alpha_p\theta_q\\
=&2e_{ipq}(\theta\cdot\nabla)\Delta\partial_q\dot c_{pjk\ell}\theta_j\theta_k\alpha_\ell+2e_{ipq}(\theta\cdot\nabla)\Delta\partial_j\dot c_{pjk\ell}\theta_k\alpha_\ell\theta_q-2c_s^2e_{ipq}(\theta\cdot\nabla)\Delta\partial_q\dot \rho\alpha_p.
\end{split}
\]

Take $\theta=\mathbf{e}_1$ and $\alpha=\mathbf{e}_2$ in the above equation.
\begin{equation} \label{SS121}
\begin{split}
    4e_{123} \partial^2_1 \partial_3 \partial_j \dot c_{2j12} + & 4e_{132} \partial^2_1 \partial_2 \partial_j \dot c_{3j12} \\
    & = 2e_{123} \partial_1 \partial_3 \Delta \dot c_{2112} + 2e_{132} \partial_1 \Delta \partial_2 \dot c_{3112} - 2c_s^2 e_{123} \partial_1 \partial_3 \Delta \dot \rho,
\end{split}
\end{equation}
\begin{equation}
\begin{split}
    4e_{213} \partial_1^2 \partial_3 \partial_j \dot c_{1j12} + & 4e_{231} \partial_1^3 \partial_j \dot c_{3j12} + e_{231} \Delta^2 \dot c_{3112} \\
    & = 2e_{213} \partial_1 \Delta \partial_3 \dot c_{1112} + 2e_{231} \partial_1 \Delta \partial_1 \dot c_{3112} + 2e_{231} \partial_1 \Delta \partial_j \dot c_{3j12},
\end{split}
\end{equation}
\begin{equation} \label{SS123}
\begin{split}
    4e_{312} \partial_1^2 \partial_2 \partial_j \dot c_{1j12} + & 4e_{321} \partial_1^3 \partial_j \dot c_{2j12} + e_{321} \Delta^2 \dot c_{2112} - c_s^2e_{321} \Delta^2 \dot \rho \\
    &= 2e_{312} \partial_1 \Delta \partial_2 \dot c_{1112} + 2e_{321} \partial_1 \Delta \partial_1 \dot c_{2112} \\
    & \quad + 2e_{321} \partial_1 \Delta \partial_j \dot c_{2j12} - 2c_s^2 e_{321} \partial_1^2 \Delta \dot \rho.
\end{split}
\end{equation}
We note that the above three equations are linearly dependent, and one of them is redundant.

Taking $\theta=\mathbf{e}_2$ and $\alpha=\mathbf{e}_1$, we have
\begin{equation} \label{SS211}
\begin{split}
4e_{123}\partial^2_2\partial_3\partial_j\dot c_{2j21}+&4e_{132}\partial^3_2\partial_j\dot c_{3j21}+e_{132}\Delta^2\dot c_{3221}\\
&=2e_{123}\partial_2\Delta\partial_3\dot c_{2221}+2e_{132}\partial_2\Delta\partial_2\dot c_{3221}+2e_{132}\partial_2\Delta\partial_j\dot c_{3j21},
\end{split}
\end{equation}
\begin{equation} \label{SS212}
4e_{213}\partial^2_2\partial_3\partial_j\dot c_{1j21}+4e_{231}\partial^2_2\partial_1\partial_j\dot c_{3j21}=2e_{213}\partial_2\Delta\partial_3\dot c_{1221}+2e_{231}\partial_2\Delta\partial_1\dot c_{3221}-2c_s^2e_{213}\partial_2\partial_3\Delta\dot \rho,
\end{equation}
\begin{equation} \label{SS223}
\begin{split}
    4e_{321} \partial_2^2 \partial_1 \partial_j \dot c_{2j21} + & 4e_{312} \partial_2^3 \partial_j \dot c_{1j21} + e_{312} \Delta^2 \dot c_{1221} - c_s^2 e_{312} \Delta^2 \dot \rho \\
    & = 2e_{321} \partial_2 \Delta \partial_1 \dot c_{2221} + 2e_{312} \partial_2 \Delta \partial_2 \dot c_{1221} \\
    & \quad + 2e_{312} \partial_2 \Delta \partial_j \dot c_{1j21} - 2 c_s^2 e_{312} \partial_2^2 \Delta \dot \rho.
\end{split}
\end{equation}

\subsection{Uniqueness of the parameters}
Keep in mind that all the parameters $\dot c_{ijk\ell}$ are compactly supported.
Applying $\partial_1^{-1}$ to \eqref{SS121} and $\partial_2^{-1}$ to \eqref{SS212}, we get
\[
4\partial_1\partial_3\partial_j\dot c_{2j12}-4\partial_1\partial_2\partial_j\dot c_{3j12}=2\partial_3\Delta (\dot c_{2112}-c_s^2\dot \rho)-2\Delta\partial_2\dot c_{3112}
\]
and
\[
4\partial_2\partial_3\partial_j\dot c_{1j21}-4\partial_2\partial_1\partial_j\dot c_{3j21}=2\partial_3\Delta (\dot c_{2112}-c_s^2\dot \rho)-2\Delta\partial_1\dot c_{3221},
\]
since $\dot{\mathbf{C}}$ is compactly supported.
The above two equations lead to
\[
2\partial_1\partial_3\partial_j\dot c_{2j12}-2\partial_2\partial_3\partial_j\dot c_{1j21}=\Delta(\partial_1\dot c_{3221}-\partial_2\dot c_{3112}).
\]
Inserting \eqref{SP121} and \eqref{SP211} into the above equation, we have
\[
\frac{2c_p^2}{c_p^2-c_s^2}(\partial_1\partial_3\partial_2\dot c_{2221}-\partial_2\partial_3\partial_1\dot c_{1112})=\Delta(\partial_1\dot c_{3221}-\partial_2\dot c_{3112}).
\]
Since we already have $\dot c_{2221}=\dot c_{1112}$, we now have $\Delta(\partial_1\dot c_{3221}-\partial_2\dot c_{3112})=0$ in $\mathbb{R}^3$ and consequently 
\begin{equation} \label{identity3221}
    \partial_1\dot c_{3221}-\partial_2\dot c_{3112}=0.
\end{equation}
Using again \eqref{identity1} and \eqref{identity2}, which imply that
\[
\partial_2^2(2\dot c_{1212}-c_s^2\dot \rho)+2\partial_2\partial_3\dot c_{1213}=\partial_1^2(2\dot c_{1212}-c_s^2\dot \rho)+2\partial_1\partial_3\dot c_{1223},
\]
we have
\[
(\partial^2_1-\partial_2^2)(2\dot c_{1212}-c_s^2\dot \rho)=0.
\]
So 
\[
2\dot c_{1212}-c_s^2\dot \rho=0.
\]
Now \eqref{SS123} reduces to
\[
\begin{split}
&4\partial_1^3\partial_2\dot c_{1112}+4\partial^2_1\partial_2\partial_3\dot c_{1312}-4\partial_1^3\partial_2\dot c_{2212}-4\partial_1^3\partial_3\dot c_{2312}+4\partial_1^2\partial^2_2\dot c_{1212}-4\partial^4_1\dot c_{1212}+\Delta^2\dot c_{1212}\\
=&2\Delta\partial_1\partial_2\dot c_{1112}-2\Delta\partial_1\partial_2\dot c_{2212}-2\Delta\partial_1\partial_3\dot c_{2312}.
\end{split}
\]
Recalling that $\dot c_{1112}-\dot c_{2212}=0$ and $\partial_2\dot c_{1312}-\partial_1\dot c_{2312}=0$, the above identity leads to 
\begin{equation} \label{eqc1212_1}
    2\Delta\partial_1\partial_3\dot c_{2312}+4\partial_1^2\partial^2_2\dot c_{1212}-4\partial^4_1\dot c_{1212}+\Delta^2\dot c_{1212}=0.
\end{equation}
Similar consideration applied to \eqref{SS223} gives
\begin{equation} \label{eqc1212_2}
    2\Delta\partial_2\partial_3\dot c_{1312}+4\partial_1^2\partial^2_2\dot c_{1212}-4\partial^4_2\dot c_{1212}+\Delta^2\dot c_{1212}=0.
\end{equation}
Subtracting \eqref{eqc1212_2} from \eqref{eqc1212_1}, we obtain
\[
(\partial_1^4-\partial_2^4)\dot c_{1212}=0,
\]
from which we get
\[
\dot c_{1212}=0.
\]
Then $\dot \rho=0$ and also, by \eqref{eqc1212_1} and \eqref{eqc1212_2},
\[
\Delta\partial_1\partial_3\dot c_{2312}=\Delta\partial_2\partial_3\dot c_{1312}=0
\]
Consequently
\[
\dot c_{2312}=\dot c_{3112}=0.
\]
Using again \eqref{identity1} and \eqref{identity2}, we get $\partial_1\dot c_{1112}=\partial_2\dot c_{2221}=0$ and then
\[
\dot c_{1112}=\dot c_{2221}=0.
\]
Now we already have $\dot c_{1112}=\dot c_{1212}=\dot c_{2212}=\dot c_{1312}=\dot c_{2312}=0$.
Using \eqref{SP122}, we have
\[
\dot c_{3312}=0.
\]
Summarizing above results, we have
\[
\dot c_{ij12}=0,\quad i,j=1,2,3.
\]

By the same consideration with different choices of $\theta=\mathbf{e}_i$ and $\alpha=\mathbf{e}_j$, with $i,j=1,2,3$, $i\neq j$, we can also have
\[
\dot c_{ij13}=\dot c_{ij23}=0,\quad i,j=1,2,3.
\]
The identities \eqref{PS122} and \eqref{PS132} then reduce to
\[
\begin{split}
\partial_3\partial_1\dot c_{1111}-\partial_1\partial_3\dot c_{3311}=0,\\
\partial_2\partial_1\dot c_{1111}-\partial_1\partial_2\dot c_{2211}=0.
\end{split}
\]
This yields
\[
\dot c_{1111}=\dot c_{2211}=\dot c_{3311}.
\]
Now the equation \eqref{PP1} reduces to
\[
4\partial^2_1\Delta \dot c_{1111}+\Delta^2\dot c_{1111}=4\partial^2_1\Delta \dot c_{1111}.
\]
So
\[
\Delta^2\dot c_{1111}=0,
\]
and consequently $\dot c_{2211}=\dot c_{3311}=\dot c_{1111}=0$. 
%
Similarly, we have
\[
\dot c_{2222}=\dot c_{3333}=\dot c_{2233}=0.
\]
Now we have shown that all components of $\mathbf{\dot C}$ and $\dot \rho$ vanish. This proves the injectivity of $\mathcal{A}'[\mathbf{C}_0,1]$, while we have only used $\theta,\alpha$ to be $\mathbf{e}_j$'s.

\section{Rigidity of an isotropic perturbation}\label{Rigidity}
In this section we consider the isotropic case, that is, when
\[
\mathbf{ C}={\lambda}\delta_{ij}\delta_{k\ell}+{\mu}(\delta_{ik}\delta_{j\ell}+\delta_{i\ell}\delta_{jk}),
\]
and prove Theorem \ref{thm:2-EL25}. 
Let us first go back to the linearized problem, but now we need to give it a quantitative study. For this isotropic case, we do not need to use $P$-$S$ or $S$-$P$ scattering. Also, we will only use incident waves in a single direction.

\subsection{Send $P$-wave, obverse $P$-wave} \label{subsec:pW-EL25}

By using a $P$-wave 
$
u_0=\theta H_0(c_p t - \theta \cdot x)
$ as the incident wave,
the total wave field $u$ satisfies the equation 
\begin{equation} \label{eq:e1-EL25}
	\left\{\begin{aligned}
		& (\rho \partial_t^2 - \mathcal L_{\lambda, \mu}) u = 0, && \quad \text{in} \quad \R^{3} \times \R, \\
		& u = \theta H_0(c_p t - \theta \cdot x), && \quad \text{for} \quad t \ll 0,
	\end{aligned}\right.
\end{equation}
where
\(
\mathcal L_{\lambda, \mu} u = \mu \Delta u + (\lambda + \mu) \nabla \nabla \cdot u,
\)
and $|\theta| = 1$.
By linearizing $(\rho, \lambda, \mu)$ at $(\rho_0 = 1, \lambda_0, \mu_0)$, we obtain
\begin{equation*}
	\left\{\begin{aligned}
		(\partial_t^2 - \mathcal L_{\lambda_0, \mu_0}) \dot{u} & = (\mathcal L_{\dot \lambda, \dot \mu} - \dot \rho \partial_t^2) u_0, && \quad \text{in} \quad \R^{3} \times \R, \\
		\dot{u} & = 0, && \quad \text{for} \quad t \ll 0.
	\end{aligned}\right.
\end{equation*}

We compute
\begin{align}
	\mathcal L_{\dot\lambda, \dot\mu} u_0
	& = - \nabla\cdot (\dot\lambda \delta I + 2\dot\mu (\theta\times\theta) \delta)
	= - \nabla (\dot\lambda \delta) - 2 \theta (\theta \cdot \nabla(\dot\mu \delta)) \nonumber \\
	& = - [(\nabla \dot\lambda) \delta - \dot\lambda \theta \delta'] - [2 (\theta \cdot \nabla \dot\mu) \theta \delta - 2\dot\mu \theta \delta'] \nonumber \\
	& = (\dot\lambda + 2\dot\mu) \theta \delta' - [(\nabla \dot\lambda) + 2 \theta (\theta \cdot \nabla \dot\mu)] \delta, \label{eq:lmu0-EL25}
\end{align}
and
\[
\dot \rho \partial_t^2 u_0
= \dot \rho c_p^2 \theta \delta'.
%
\]
So the single scattered wave $\dot u$ solves the following equation.
\begin{equation} \label{eq:dup-EL25}
	\left\{\begin{aligned}
		(\partial_t^2 - \mathcal L_{\lambda_0, \mu_0}) \dot{u} & = (\dot \lambda + 2\dot \mu - \dot \rho c_p^2) \theta \delta' - [\nabla \dot \lambda + 2 \theta (\theta \cdot \nabla \dot \mu)] \delta, && \quad \text{in} \quad \R^{3} \times \R, \\
		\dot{u} & = 0, && \quad \text{for} \quad t \ll 0,
	\end{aligned}\right.
\end{equation}
We consider $\nabla \cdot \dot{u}$ which encodes the compressional wave.
By taking the divergence of both sides, we have
\begin{equation} \label{eq:lmua-EL25}
	\left\{\begin{aligned}
		\square_p (\nabla \cdot \dot{u}) & = - [\Delta \dot \lambda + 2 (\theta \cdot \nabla)^2 \dot \mu] \delta \\
		& \quad + (\theta \cdot \nabla) (2\dot \lambda + 4\dot \mu - \dot \rho c_p^2) \delta' \\
		& \quad - (\dot \lambda + 2\dot \mu - \dot \rho c_p^2) \delta'', && \quad \text{in} \quad \R^{3} \times \R, \\
		\nabla \cdot \dot{u} & = 0, && \quad \text{for} \quad t \ll 0,
	\end{aligned}\right.
\end{equation}

By the discussion in Section \ref{secPP}, we can use the following ansatz
\begin{equation*}
	\nabla \cdot \dot{u}
	= w_2(x) \delta'(c_p t - \theta \cdot x) + w_1(x) \delta(c_p t - \theta \cdot x) + w_0(x,t) H_0(c_p t - \theta \cdot x),
\end{equation*}
with $w_2(x) = w_1(x) = 0$ when $\theta \cdot x \ll 0$ and $w_0(x,t) = 0$ when $t \ll 0$.
By computation, we have
\begin{align}
	\square_p (\nabla \cdot \dot{u})
	& = (\square_p w_0) H_0 + (L_p w_0 - c_p^2\Delta w_1) \delta + (2c_p^2\theta\cdot\nabla w_1- c_p^2\Delta w_2) \delta' + (2c_p^2\theta\cdot\nabla w_2) \delta'', \label{eq:PP-EL25}
\end{align}
where
\(
L_p:= 2c_p (\partial_t + c_p \theta \cdot \nabla).
\)
Note that when acting on functions independent of $t$ such as $w_1$ and $w_2$, $L_p$ will be the same as $2c_p^2 \theta \cdot \nabla$, so we also use the notation
\[
L_\theta := \theta \cdot \nabla
\]
when it is appropriate.
Comparing different singularities in \eqref{eq:lmua-EL25} and \eqref{eq:PP-EL25} gives
\begin{equation*}
	\left\{\begin{aligned}
		\square_p w_0 & = 0, && \text{in} \ \ \{ c_p t - \theta \cdot x > 0 \}, \\
		L_p w_0 - c_p^2\Delta w_1 & = - [\Delta \dot \lambda + 2 (\theta \cdot \nabla)^2 \dot \mu] =: F_0, && \text{on} \ \ \{ c_p t - \theta \cdot x = 0 \}, \\
		2c_p^2\theta\cdot\nabla w_1 - c_p^2\Delta w_2 & = (\theta \cdot \nabla) (2\dot \lambda + 4\dot \mu - \dot \rho c_p^2) =: F_1, && \text{in} \ \mathbb{R}^3, \\
		2c_p^2\theta\cdot\nabla w_2 & = - (\dot \lambda + 2\dot \mu - \dot \rho c_p^2) =: F_2, && \text{in} \ \mathbb{R}^3.
	\end{aligned}\right.
\end{equation*}
The last two equations on $\mathbb{R}^3$ are transport equations for $w_1(x),w_2(x)$.
Note that $F_1$ and $F_2$ are compactly supported, so we can successively solve the transport equations to determine $w_2,w_1$.
That means that we can define the inverse of $L_\theta$, e.g.,~$L_\theta^{-1}$ and write
\[
w_2=\frac{1}{2c_p^2}L_\theta^{-1}F_2,\quad w_1=\frac{1}{2c_p^2}L_\theta^{-1}(c_p^2\Delta w_2+F_1).
\]
So we can simplify the system as follows,
\begin{equation} \label{eq:dpw-EL25}
	\left\{\begin{aligned}
		\square_p w_0 & = 0, && \text{in} \ \ \{ c_p t - \theta \cdot x > 0 \}, \\
		L_p w_0 & = \sum_{r=0}^2 \left(\frac \Delta 2 L_\theta^{-1}\right)^r F_r, && \text{on} \ \ \{ c_p t - \theta \cdot x = 0 \}.
	\end{aligned}\right.
\end{equation}

Recall $\Omega := \{ x \in \mathbb{R}^3 : |x|<1 \}$.
Denote
\[
\Sigma_{T,p} := \Sigma_{T,p,\theta} := \big( \partial \Omega \times (-\infty,T] \big) \cap \{ c_p t - \theta \cdot x > 0 \}
\]
and
\[
\mathcal D_p(\theta) := \mathcal D_p(\dot \lambda, \dot \mu, \dot \rho;\theta) := \sum_{r=0}^2 \left(\frac \Delta 2 L_\theta^{-1}\right)^r F_r.
\]
Applying \cite[Proposition 1.4]{ma2023fixed} to \eqref{eq:dpw-EL25}, we can have
\begin{equation} \label{eq:est1-EL25}
	\norm{\mathcal D_p(\theta)}_{L^2(\Omega)} \lesssim \norm{w_0(\cdot,\cdot;\theta,\theta)}_{H^2(\Sigma_{T,p})}
    = \norm{\nabla \cdot \dot u(\cdot,\cdot;\theta,\theta)}_{H^2(\Sigma_{T,p})}.
\end{equation}
Here $w_0(\cdot,\cdot;\theta,\theta)$ signifies $w_0$ with an emphasis on the dependence of $w_0$ on $\theta$.
The last equal sign in \eqref{eq:est1-EL25} is due to the fact that $\nabla \cdot \dot{u} = w_0$ on $\Sigma_{T,p}$.

We directly compute
\begin{align}
    L_\theta^2 \mathcal D_p(\theta)
	& = - (L_\theta^2\Delta \dot \lambda + 2 L_\theta^4 \dot \mu) + \frac \Delta 2  L^2_\theta (2\dot \lambda + 4\dot \mu - \dot \rho c_p^2) -\frac{\Delta^2}{4}(\dot \lambda + 2\dot \mu - \dot \rho c_p^2) \nonumber \\
	& = - \frac {\Delta^2} 4 \dot \lambda + 2 \left( L_\theta^2\Delta - L_\theta^4 - \frac {\Delta^2} 4\right) \dot \mu + c_p^2 \left( \frac{\Delta^2}{4} - \frac {L_\theta^2\Delta} 2 \right) \dot \rho. \label{eq:Dp-EL25}
\end{align}

\subsection{Send $S$-wave, observe $S$-wave} \label{subsec:sW-EL25}

By using a shear wave $
u_0=\alpha H_0(c_s t - \theta \cdot x)
$ as the incident wave, we have
\begin{equation} \label{eq:e1s-EL25}
	\left\{\begin{aligned}
		& (\rho \partial_t^2 - \mathcal L_{\lambda, \mu}) u = 0, && \quad \text{in} \quad \R^{3} \times \R, \\
		& u = \alpha H_0(c_s t - \theta \cdot x), && \quad \text{for} \quad t \ll 0,
	\end{aligned}\right.
\end{equation}
with $|\alpha| = 1$ and $\alpha \perp \theta$.
Similar to Section \ref{subsec:pW-EL25}, we now have
\begin{equation} \label{eq:dus-EL25}
	\left\{\begin{aligned}
		(\partial_t^2 - \mathcal L_{\lambda_0, \mu_0}) \dot{u} & = (\dot \mu - \dot \rho c_s^2) \alpha \delta'  - (\theta (\alpha\cdot \nabla \dot \mu) + \alpha (\theta\cdot \nabla \dot \mu)) \delta, && \quad \text{in} \quad \R^{3} \times \R, \\
		\dot{u} & = 0, && \quad \text{for} \quad t \ll 0.
	\end{aligned}\right.
\end{equation}
Consider $\nabla \times \dot{u}$ which encodes the $S$ wave.
By taking the curl of both sides, we obtain
\begin{equation} \label{eq:lmub-EL25}
	\left\{\begin{aligned}
		\square_s (\nabla \times \dot{u}) & = - \big[ (\theta \times \nabla) (\alpha \cdot \nabla) + (\alpha \times \nabla) (\theta \cdot \nabla) \big] \dot \mu \delta \\
		& \quad + \big[ (\theta \times \alpha) (\theta \cdot \nabla) \dot \mu - (\alpha \times \nabla) (\dot{\mu} - \dot{\rho} c_s^2) \big] \delta' \\
		& \quad + (\alpha \times \theta) (\dot{\mu} - \dot{\rho} c_s^2) \delta'', && \quad \text{in} \quad \R^{3} \times \R, \\
		\nabla \times \dot{u} & = 0, && \quad \text{for} \quad t \ll 0,
	\end{aligned}\right.
\end{equation}

By the discussions in Section \ref{sectSS}, we use the ansatz
\begin{equation*}
	\nabla \times \dot{u}
	= w_2(x) \delta'(c_s t - \theta \cdot x) + w_1(x) \delta(c_s t - \theta \cdot x) + w_0(x,t) H(c_s t - \theta \cdot x),
\end{equation*}
where $w_2$, $w_1$ and $w_0$ are vector-valued functions, and $w_2(x) = w_1(x) = 0$ when $\theta \cdot x \ll 0$ and $w_0(x,t) = 0$ when $t \ll 0$.
We compute
\begin{align}
	\square_s (\nabla \times \dot u)
	& = (\square_s w_0) H_0 + (L_s w_0 - c_s^2\Delta w_1) \delta + (2c_s^2\theta\cdot\nabla w_1 - c_s^2\Delta w_2) \delta' + (2c_s^2\theta\cdot\nabla w_2) \delta'', \label{eq:lmu2-EL25}
\end{align}
where
\(
L_s
= 2c_s (\partial_t + c_s \theta \cdot \nabla).
\)
Note that when acting on functions independent of $t$ such as $w_1$ and $w_2$, $L_s$ will be the same as $2c_s^2 \theta \cdot \nabla=2c_s^2L_\theta$.
Comparing different singularities in \eqref{eq:lmub-EL25} and \eqref{eq:lmu2-EL25} gives
\begin{equation*}
	\left\{\begin{aligned}
		\square_s w_0 & = 0, && \text{in} \ \ \{ c_s t - \theta \cdot x > 0 \}, \\
		L_s w_0 - c_s^2\Delta w_1 & = - \big[ (\theta \times \nabla) (\alpha \cdot \nabla) + (\alpha \times \nabla) (\theta \cdot \nabla) \big] \dot \mu =: G_0, && \text{on} \ \ \{ c_s t - \theta \cdot x = 0 \}, \\
		2c_s^2\theta\cdot\nabla w_1 - c_s^2\Delta w_2 & = \big[ (\theta \times \alpha) (\theta \cdot \nabla) \dot \mu - (\alpha \times \nabla) (\dot{\mu} - \dot{\rho} c_s^2) \big] =: G_1, && \text{on} \ \ \mathbb{R}^3, \\
		2c_s^2\theta\cdot\nabla w_2 & = (\alpha \times \theta) (\dot{\mu} - \dot{\rho} c_s^2) =: G_2, && \text{on} \ \ \mathbb{R}^3.
	\end{aligned}\right.
\end{equation*}
Similar to the $P$-wave case, we can simplify the system as follows,
\begin{equation} \label{eq:dpw2-EL25}
	\left\{\begin{aligned}
		\square_s w_0 & = 0, && \text{in} \ \ \{ c_s t - \theta \cdot x > 0 \}, \\
		L_s w_0 & = \sum_{r=0}^2 \left(\frac \Delta 2 L_\theta^{-1}\right)^r G_r, && \text{on} \ \ \{ c_s t - \theta \cdot x = 0 \}.
	\end{aligned}\right.
\end{equation}

Denote
\[
\Sigma_{T,s} := \Sigma_{T,s,\theta} := \big( \partial \Omega \times (-\infty,T] \big) \cap \{ c_s t - \theta \cdot x > 0 \}
\]
and
\[
\mathcal D_s(\theta, \alpha) := \mathcal D_s( \dot \lambda, \dot \mu, \dot \rho;\theta,\alpha) := \sum_{r=0}^2 \left(\frac \Delta 2 L_\theta^{-1}\right)^r G_r.
\]
Using \cite[Proposition 1.4]{ma2023fixed} to \eqref{eq:dpw2-EL25} we can have
\begin{equation} \label{eq:est2-EL25}
	\norm{\mathcal D_s(\theta, \alpha)}_{L^2(\Omega)} \lesssim \norm{w_0 (\cdot,\cdot;\theta,\alpha)}_{H^2 (\Sigma_{T,s})} = \norm{\nabla \times \dot{u}(\cdot,\cdot;\theta,\alpha)}_{H^2 (\Sigma_{T,s})}.
\end{equation}
The last equal sign in \eqref{eq:est2-EL25} is due to the fact that $\nabla \times \dot{u} = w_0$ on $\Sigma_{T,s}$.
We can directly compute
\begin{align}
	L_\theta^2 \mathcal D_s (\theta, \alpha)
	& = - L_\theta^2\big[ (\theta \times \nabla) (\alpha \cdot \nabla) + (\alpha \times \nabla) (\theta \cdot \nabla) \big] \dot \mu \nonumber \\
	& \quad + \frac \Delta 2 L_\theta \big[ (\theta \times \alpha) (\theta \cdot \nabla) \dot \mu - (\alpha \times \nabla) (\dot{\mu} - \dot{\rho} c_s^2) \big] - \frac{\Delta^2}{4} \big[ (\alpha \times \theta) (\dot{\mu} - \dot{\rho} c_s^2) \big]. \label{eq:Ds-EL25}
\end{align}

Now we choose a coordinate frame $\{\bm{e_1} = \theta,\, \bm{e_2} = \alpha,\, \bm{e_3} = \theta \times \alpha \}$.
Rewriting \eqref{eq:Dp-EL25} and \eqref{eq:Ds-EL25} under this coordinates system gives
\begin{equation} \label{eq:Dp3-EL25}
    \partial_1^2 \mathcal D_p(\bm{e_1})
	= - \frac {\Delta^2} 4 \dot \lambda + 2 \left( \partial_1^2 \Delta - \partial_1^4 - \frac {\Delta^2} 4\right) \dot \mu + \left( \frac{\Delta^2}{4} - \frac \Delta 2 \partial_1^2 \right) \dot \rho c_p^2,
\end{equation}
and
\begin{align}
	\partial_1^2 \mathcal D_s (\bm{e_1}, \bm{e_2})
    & = \bm{e_1} \left[ \partial^2_{13} \frac \Delta 2 \dot{\rho} c_s^2 - \left(\partial^4_{1113} + \frac \Delta 2 \partial^2_{13}\right) \dot \mu \right] + \bm{e_2} \partial^4_{1123} \dot \mu \nonumber \\
	& \quad + \bm{e_3} \left[ \left(\partial_{1}^4 + \Delta \partial_1^2 + \frac {\Delta^2} 4 - \partial_{12}^2\right) \dot \mu - \left(\partial_1^2 + \frac \Delta 2\right) \frac \Delta 2 \dot{\rho} c_s^2 \right]. \label{eq:Ds3-EL25}
\end{align}

We first recover $\dot \mu$.
From \eqref{eq:Ds3-EL25}, we see the second component of the vector $\partial_1^2 \mathcal D_s(\theta, \alpha)$ equals to $\partial^4_{1123} \dot \mu$, and actually the second component of the vector $\mathcal D_s(\theta, \alpha)$ equals to $\partial^2_{23} \dot \mu$.
Hence, by \eqref{eq:est2-EL25} we can have
\begin{equation*}
	\norm{\dot \mu}_{L^2(\Omega)}
	\lesssim \norm{\partial^2_{23} \dot \mu}_{L^2(\Omega)}
	\lesssim \norm{\mathcal D_s(\bm{e_1}, \bm{e_2})}_{L^2(\Omega)}
	\lesssim \norm{\nabla \times \dot{u}(\cdot,\cdot;\bm{e_1}, \bm{e_2})}_{H^2 (\Sigma_{T,s})}.
\end{equation*}
Therefore, when either acting differential operators $\partial_t + c_s \theta \cdot \nabla$ or $\alpha \cdot \nabla$ multiple times to \eqref{eq:dpw2-EL25} and using similar arguments as above, we can also obtain for any multi-index $\beta$,
\begin{equation} \label{eq:est3-EL25}
	\norm{\partial^\beta \dot \mu}_{L^2(\Omega)}
	\lesssim \norm{\nabla_{x,t}^{|\beta|} \nabla \times \dot{u}(\cdot,\cdot;\bm{e_1}, \bm{e_2})}_{H^2(\Sigma_{T,s})}.
\end{equation}

Next we recover $\dot \rho$.
From \eqref{eq:Ds3-EL25}, we see the third component of $\partial_1^2 \mathcal D_s (\bm{e_1}, \bm{e_2})$ equals to $(\partial_{1}^4 + \Delta \partial_1^2 + \frac {\Delta^2} 4 - \partial_{12}^2) \dot \mu - (\partial_1^2 + \frac \Delta 2) \frac \Delta 2 \dot{\rho} c_s^2$, so by \eqref{eq:est2-EL25} we have
\begin{align*}
	\norm{\nabla_{x,t}^2 \nabla \times \dot{u}(\cdot,\cdot;\bm{e_1}, \bm{e_2})}_{H^2 (\Sigma_{T,s})}
    & \gtrsim \norm{\partial_1^2 \mathcal D_s (\bm{e_1}, \bm{e_2})}_{L^2(\Omega)} \\
	& \gtrsim \left\|\left(\partial_{1}^4 + \Delta \partial_1^2 + \frac {\Delta^2} 4 - \partial_{12}^2\right) \dot \mu - \left(\partial_1^2 + \frac \Delta 2\right) \frac \Delta 2 \dot{\rho} c_s^2\right\|_{L^2(\Omega)} \nonumber \\
    & \gtrsim \norm{(2\partial_1^2 + \Delta) \Delta \dot{\rho}}_{L^2(\Omega)} - \norm{\dot \mu}_{H^4(\Omega)}.
\end{align*}
Note that the operator $-(2\partial_1^2 + \Delta)$ is elliptic and $\dot \rho$ has zero trace on $\partial \Omega$, so by the interior regularity of elliptic operator, together with the Ponicar\'e inequality, we have
\[
\norm{(2\partial_1^2 + \Delta) \Delta \dot{\rho}}_{L^2(\Omega)}
\lesssim \norm{\Delta \dot{\rho}}_{H^2(\Omega)}
\lesssim \norm{\dot{\rho}}_{H^4(\Omega)}.
\]
Moreover, by \eqref{eq:est3-EL25} we have
\(
\norm{\dot \mu}_{H^4(\Omega)} \lesssim \norm{\nabla_{x,t}^4 \nabla \times \dot{u}(\cdot,\cdot;\bm{e_1}, \bm{e_2})}_{H^2(\Sigma_{T,s})},
\)
thus we can continue the computation,
\begin{equation} \label{eq:est4-EL25}
	\norm{\nabla_{x,t}^2 \nabla \times \dot{u}(\cdot,\cdot;\bm{e_1}, \bm{e_2})}_{H^2 (\Sigma_{T,s})}
    \gtrsim \norm{\dot \rho}_{H^4(\Omega)} - \norm{\nabla_{x,t}^4 \nabla \times \dot{u}(\cdot,\cdot;\bm{e_1}, \bm{e_2})}_{H^2(\Sigma_{T,s})}.
\end{equation}

Finally we recover $\dot \lambda$.
Similarly, applying $(\partial_t + c_p \theta \cdot \nabla)^2$ to \eqref{eq:dpw-EL25} and using \eqref{eq:Dp3-EL25}, we have
\begin{align*}
	& \, \norm{\nabla_{x,t}^2 \nabla \cdot \dot u(\cdot,\cdot;\bm{e_1}, \bm{e_1})}_{H^2(\Sigma_{T,p})} \\
	\gtrsim & \, \left\| {-}\left(\frac \Delta 2\right)^2 \dot \lambda + 2 \left[ \Delta \partial_1^2 - \partial_1^4 - \left(\frac \Delta 2\right)^2 \right] \dot \mu + \left[ \left(\frac \Delta 2\right)^2 - \frac \Delta 2 \partial_1^2 \right] \dot \rho c_p^2 \right\|_{L^2(\Omega)} \\
	\gtrsim & \, \norm{\Delta^2 \dot \lambda}_{L^2(\Omega)} - \sum_{|\beta| = 4} \norm{\partial^\beta \dot \mu}_{L^2(\Omega)} - \sum_{|\beta| = 4} \norm{\partial^\beta \dot \rho}_{L^2(\Omega)}\\
    \gtrsim & \, \norm{\dot \lambda}_{H^4(\Omega)} - \norm{\dot \mu}_{H^4(\Omega)} - \norm{\dot \rho}_{H^4(\Omega)}.
\end{align*}
Combining this with \eqref{eq:est3-EL25} and \eqref{eq:est4-EL25}, we finally have
\begin{align}
    & \, \norm{\dot \lambda}_{H^4(\Omega)} + \norm{\dot \mu}_{H^4(\Omega)} + \norm{\dot \rho}_{H^4(\Omega)} \nonumber \\
	\lesssim & \, \norm{\nabla_{x,t}^2 \nabla \cdot \dot u(\cdot,\cdot;\bm{e_1}, \bm{e_1})}_{H^2(\Sigma_{T,p})} + \sum_{j \leq 4} \norm{\nabla_{x,t}^j \nabla \times \dot{u}(\cdot,\cdot;\bm{e_1}, \bm{e_2})}_{H^2(\Sigma_{T,s})} \nonumber \\
    \lesssim & \, \norm{\nabla_{x,t}^3 \dot u(\cdot,\cdot;\bm{e_1}, \bm{e_1})}_{H^2(\Sigma_{T,p})} + \norm{\nabla_{x,t}^5 \dot{u}(\cdot,\cdot;\bm{e_1}, \bm{e_2})}_{H^2(\Sigma_{T,s})}. \label{eq:est5-EL25}
\end{align}
Remember that $\dot u(\cdot,\cdot;\bm{e_1}, \bm{e_1})$ is the linearization of the solution of $(\rho \partial_t^2 - \mathcal L_{\lambda, \mu}) u = 0$ the with incident $P$ wave $u_0 = \bm{e_1} H_0(c_p t - \bm{e_1} \cdot x)$,
and $\dot u(\cdot,\cdot;\bm{e_1}, \bm{e_2})$ is the linearization of the solution with the incident $S$ wave $u_0 = \bm{e_2} H_0(c_s t - \bm{e_1} \cdot x)$.
We obtained the Lipschitz stability.

\subsection{The nonlinear map}

To derive uniqueness and stability for the local rigidity problem, we shall use the following abstract result from \cite[Theorem 2]{SU12}.

\begin{proposition} \label{prop:StUh-EL25}
	Let $\mathcal{X}_j$, $\mathcal{Y}_j$ with $j=1,2,3$ be Banach spaces with $\mathcal{X}_3 \subset \mathcal{X}_1 \subset \mathcal{X}_2$ and $\mathcal{Y}_3 \subset \mathcal{Y}_2 \subset \mathcal{Y}_1$, such that the following interpolation estimates hold:
	\begin{equation} \label{eq:Con1}
		\norm{f}_{\mathcal{X}_1} \lesssim \norm{f}_{\mathcal{X}_2}^{\mu_1} \norm{f}_{\mathcal{X}_3}^{1-\mu_1}, \quad \norm{g}_{\mathcal{Y}_2} \lesssim \norm{g}_{\mathcal{Y}_1}^{\mu_2} \norm{g}_{\mathcal{Y}_3}^{1-\mu_2}, \quad \mu_1, \mu_2 \in (0,1], \; \; \mu_1 \mu_2 >1/2.
	\end{equation}
	Let $\mathcal{A}: \mathcal{V}_1 \to \mathcal{Y}_1$ be a nonlinear map where $\mathcal{V}_1 \subset \mathcal{X}_1$ is an open subset of $ \mathcal{X}_1$. Consider $f_0\in \mathcal{V}_1$ and assume that 
	\begin{equation} \label{eq:Con2}
		\mathcal{A}(f) = \mathcal{A}(f_0) + A_{f_0}(f-f_0) + R_{f_0}(f), \quad \norm{R_{f_0}(f)}_{\mathcal{Y}_1} \leq C(f_0) \norm{f-f_0}^2_{\mathcal{X}_1}
	\end{equation}
	holds for all $f$ in some neighborhood of $f_0$ in $\mathcal{V}_1$.
	Here $A_{f_0}$ stands for the Fr\'echet derivative of $\mathcal{A}$ at $f_0$.
	In addition, suppose  that
	\begin{equation} \label{eq:Con3}
		\norm{h}_{\mathcal{X}_2} \leq C \norm{A_{f_0} h}_{\mathcal{Y}_2}, \quad h \in \mathcal{X}_1
	\end{equation}
	and
	\begin{equation} \label{eq:Con4}
		\norm{A_{f_0} h}_{\mathcal{Y}_3}\leq C \norm{h}_{\mathcal{X}_3}, \quad h\in \mathcal{X}_3.
	\end{equation}
	Then for any $L>0$ there exists $\epsilon>0$, so that for any $f$ with
	\begin{equation} \label{cond:f_bounds}
		\norm{f-f_0}_{\mathcal{X}_1} \leq \epsilon, \quad \norm{f}_{\mathcal{X}_3}\leq L,
	\end{equation}
	one has the conditional stability estimate
	\[
	\norm{f-f_0}_{_{\mathcal{X}_1}}\leq CL^{2-\mu_1-\mu_2} \norm{\mathcal{A}(f)-\mathcal{A}(f_0)}^{\mu_1\mu_2}_{\mathcal{Y}_1}.
	\]
	In particular, if $\mathcal A(f) = \mathcal A(f_0)$ for some $f$ satisfying \eqref{cond:f_bounds}, then $f=f_0$.
\end{proposition}

We are ready to prove Theorem \ref{thm:2-EL25}.

\begin{proof}[Proof of Theorem \ref{thm:2-EL25}]
    Denote $\kappa := (\lambda, \mu, \rho)$ and $\kappa_0 := (\lambda_0, \mu_0, 1)$.
    Abbreviate $P_\kappa := \rho \partial_t^2 - \mathcal L_{\lambda, \mu}$ and $P_{\kappa_0} := \partial_t^2 - \mathcal L_{\lambda_0, \mu_0}$.
    Let $P_\kappa u = 0$, $P_{\kappa_0} u_0 = 0$ with initial conditions as in either \eqref{eq:e1-EL25} or \eqref{eq:e1s-EL25}.
    We use $u^p$ to indicate the wave corresponding to the initial condition $u^p = \theta H_0(c_p t - \theta \cdot x)$ when $t \ll 0$, and use $u^s$ to indicate the wave corresponding to the initial condition $u^s = \alpha H_0(c_s t - \theta \cdot x)$ when $t \ll 0$.
    The superscripts $p$ and $s$ represent compressional and shear waves.
    
    Denote the following nonlinear map
    \[
    \mathcal A \colon \kappa \ \mapsto \ (u^p, u^s)|_{\Omega^c \times (-\infty,T]}.
    \]
    We want to apply Proposition \ref{prop:StUh-EL25} to the operator $\mathcal A$.
    To that end, we let $\dot{u}$ satisfy $P_{\kappa_0} \dot{u} = -P_{\dot \kappa} u_0$ with the initial condition $\dot{u} = 0$ when $t \ll 0$.
    Let $r = u - u_0 - \dot{u}$ with the initial condition $r = 0$ when $t \ll 0$.
    To apply Proposition \ref{prop:StUh-EL25}, we need to check the following
    \begin{align}
    	& \norm{(r^p, r^s)}_{H^{-k}((-\infty,T); H^{s+2}(\Omega^c))} \leq C \norm{\kappa - \kappa_0}_{H^{s_0}(\Omega)}^2, \label{eq:Con1-EL25} \\
    	& \norm{\dot \kappa}_{L^2(\Omega)} \leq C \norm{(\dot{u}^p, \dot{u}^s)}_{H^{10}((-\infty,T); H^{10}(\Omega^c))}, \label{eq:Con2-EL25} \\
    	& \norm{(\dot{u}^p, \dot{u}^s)}_{H^\alpha((-\infty,T); H^\beta(\Omega^c))} \leq C \norm{\dot \kappa}_{H^N(\Omega)}. \label{eq:Con3-EL25}
    \end{align}
    The corresponding function spaces are specified as follows
    \begin{alignat*}{3}
    	\mathcal X_1 & = H^{s_0}(\Omega), & \mathcal X_2 & = L^2(\Omega), & \mathcal X_3 & = H^\alpha((-\infty,T); H^\beta(\Omega^c)), \\
    	\mathcal Y_1 & = H^{-k} H^{s+2}(\Omega^c), \quad & \mathcal Y_2 & = H^{10}((-\infty,T); H^{10}(\Omega^c)), \quad & \mathcal Y_3 & = H^N(\Omega),
    \end{alignat*}
    where $\alpha$, $\beta$ and $N$ are chosen large enough such that the inclusion relation in Proposition \ref{prop:StUh-EL25} is satisfied, and $k \geq 0$ be an integer, and let $-s_0 \leq s \leq s_0$.
    The conditions \eqref{eq:Con1-EL25}--\eqref{eq:Con3-EL25} correspond to \eqref{eq:Con2}--\eqref{eq:Con4}, respectively.
    We verify conditions \eqref{eq:Con1-EL25}-\eqref{eq:Con3-EL25} one by one.\\
    
    \noindent\textbf{Condition \eqref{eq:Con1-EL25}:}
    $r$ satisfies
    \[
    P_{\kappa_0} r
    = P_{\kappa_0} (u - u_0 - \dot{u})
    = P_{\kappa_0 - \kappa} u + P_{\kappa - \kappa_0} u_0
    = P_{\kappa - \kappa_0} (u_0 - u).
    \]
    By Helmholtz decomposition we could divide $R$ as
    \(
    r = \nabla \phi + \nabla \times \vec \psi,
    \)
    then
    \[
    \square_p (\Delta \phi) = \nabla \cdot \big(P_{\kappa - \kappa_0} (u_0 - u) \big), \quad \text{and} \quad
    \square_s (\nabla \times \nabla \times \vec \psi) = \nabla \times \big( P_{\kappa - \kappa_0} (u_0 - u) \big).
    \]
    In the rest of the proof, we abbreviate the $H^b(\R^3)$-valued Sobolev spaces $H^a((-\infty,T); H^b(\R^3))$ as $H^a H^b$ for short.
    By \cite[Lemma A.4]{ma2023fixed}, we have
    \begin{align*}
    	\norm{(\nabla \cdot r, \nabla \times r)}_{H^{-k} H^{s+1}}
    	& = \norm{(\Delta \phi, \nabla \times \nabla \times \vec \psi)}_{H^{-k} H^{s+1}} \\
    	& \lesssim \norm{P_{\kappa - \kappa_0} (u_0 - u)}_{H^{-k} H^{s+1}} \\
    	& \lesssim \norm{\kappa - \kappa_0}_{L^\infty(\Omega)} (\norm{u_0 - u}_{H^{-k+2} H^{s+1}} + \norm{u_0 - u}_{H^{-k} H^{s+3}}).
    \end{align*}
    Because
    \[
    \norm{(\nabla \cdot r, \nabla \times r)}_{L^2(\R^3)}
    = \norm{(\xi \cdot \hat{r}, \xi \times \hat{r})}_{L^2(\R^3)}
    = \big\lVert |\xi| |\hat{r}| \big\rVert_{L^2(\R^3)}
    \simeq \norm{\nabla r}_{L^2(\R^3)},
    \]
    so
    \[
    \norm{(\nabla \cdot r, \nabla \times r)}_{H^{-k} H^{s+1}}
    \simeq \norm{\nabla r}_{H^{-k} H^{s+1}}
    \simeq \norm{r}_{H^{-k} H^{s+2}},
    \]
    thus
    \begin{equation} \label{eq:Co1-EL25}
    	\norm{r}_{H^{-k} H^{s+2}}
    	\lesssim \norm{\kappa - \kappa_0}_{L^\infty(\Omega)} (\norm{u_0 - u}_{H^{-k+2} H^{s+1}} + \norm{u_0 - u}_{H^{-k} H^{s+3}}).
    \end{equation}
    Once again, because
    \(
    P_{\kappa_0} (u_0 - u)
    = -P_{\kappa_0} u
    = P_{\kappa - \kappa_0} u,
    \)
    so similar to the derivation of \eqref{eq:Co1-EL25}, we can have
    \begin{equation} \label{eq:Co2-EL25}
    	\norm{u_0 - u}_{H^{-k} H^{s+1}}
    	\lesssim \norm{\kappa - \kappa_0}_{L^\infty(\Omega)} (\norm{u}_{H^{-k+2} H^{s+1}} + \norm{u}_{H^{-k} H^{s+3}}).
    \end{equation}
    Combining \eqref{eq:Co1-EL25} and \eqref{eq:Co2-EL25}, we obtain
    \begin{equation} \label{eq:Co3-EL25}
    	\norm{r}_{H^{-k} H^{s+2}}
    	\lesssim \norm{\kappa - \kappa_0}_{H^{s_0}(\Omega)}^2 (\norm{u}_{H^{-k} H^{s+5}} + \norm{u}_{H^{-k+2} H^{s+3}} + \norm{u}_{H^{-k+4} H^{s+1}}),
    \end{equation}
    where $s_0 > \lceil 3/2 \rceil$ so that $H^{s_0}(\Omega) \hookrightarrow L^\infty(\Omega)$.
    This implies \eqref{eq:Con1-EL25}.\\

    \noindent\textbf{Condition \eqref{eq:Con2-EL25}:}
    This condition is the Lipschitz stability of the linearized problem, and it is implied by \eqref{eq:est5-EL25}.
    Indeed, from \eqref{eq:est5-EL25} we have
    \begin{align*}
    	\norm{\dot \kappa}_{L^2(\Omega)}
        & \lesssim \norm{\dot \lambda}_{L^2(\Omega)} + \norm{\dot \mu}_{L^2(\Omega)} + \norm{\dot \rho}_{L^2(\Omega)} \\
        & \lesssim \norm{\nabla_{x,t}^3 \dot u(\cdot,\cdot;\bm{e_1}, \bm{e_1})}_{H^2(\Sigma_{T,p})} + \norm{\nabla_{x,t}^5 \dot{u}(\cdot,\cdot;\bm{e_1}, \bm{e_2})}_{H^2(\Sigma_{T,s})} \\
        & \lesssim \norm{(\dot{u}^p, \dot{u}^s)}_{H^{10}((-\infty,T); H^{10}(\Omega^c))}.
    \end{align*}
   ~\\
    
    \noindent\textbf{Condition \eqref{eq:Con3-EL25}:}
    This condition is an \textit{a priori} estimate for the linearized elastic system.
    We know that $\dot{u}$ satisfies $P_{\kappa_0} \dot{u} = -P_{\dot \kappa} u_0$, thus similar to the derivation of \eqref{eq:Co1-EL25} we have
    \begin{equation} \label{eq:Co4-EL25}
    	\norm{\dot{u}}_{H^{-k} H^{s+2}}
    	\lesssim \norm{\dot \kappa}_{L^\infty(\Omega)} (\norm{u_0}_{H^{-k+2} H^{s+1}} + \norm{u_0}_{H^{-k} H^{s+3}})
    	\lesssim \norm{\dot \kappa}_{H^{s_0}(\Omega)}.
    \end{equation}
    which further indicates \eqref{eq:Con3-EL25}.
    
    Therefore, by Proposition \ref{prop:StUh-EL25}, we obtain the following local H\"older inverse stability:
    \begin{equation*}
    	\norm{(\rho - \rho_0, \lambda - \lambda_0, \mu - \mu_0)}_{H^{s_0}(\Omega)} \leq C \norm{(u^p - u_0^p, u^s - u_0^s)}_{H^{-k}((-\infty,T]; H^{s+2}(\Omega^c))}^\mu.
    \end{equation*}
    The proof is complete.
\end{proof}

\section*{Acknowledgements}

The research of M.L. was partially supported by PDE-Inverse project of the European Research Council of the European Union, the FAME and Finnish Quantum flagships and the grant 336786 of the Research Council of Finland. 
The research of S.M. is partially supported by the NSFC (No.~12301540).
L.O. was supported by the European Research Council
of the European Union, grant 101086697 (LoCal),
and the Research Council of Finland, grants 347715,
353096 (Centre of Excellence of Inverse Modelling and Imaging)
and 359182 (Flagship of Advanced Mathematics for Sensing Imaging and Modelling).
Views and opinions expressed are those of the authors only and do not
necessarily reflect those of the European Union or the other funding
organizations.
M.S. was partly supported by the Research Council of Finland (Centre of Excellence in Inverse Modelling and Imaging and FAME Flagship, grants 353091 and 359208).
J.Z. is supported by National Key Research and Development Programs of China (No.~2023YFA1009103), NSFC (No.~12471396), Science and Technology Commission of Shanghai Municipality (23JC1400501).

\vspace{.3cm}
\noindent \textbf{Data Availability:} The authors shall permit all the data underlying the findings of this manuscript to be shared by any researchers or groups who are interested in the article.

\vspace{.3cm}
\noindent \textbf{Conflict of Interest:} The authors declare that there is no conflict of interest regarding the publication of this paper.


{

}

\end{document}